\documentclass[12pt]{amsart}
\usepackage[cmtip,all]{xy}
\usepackage{etex}

\usepackage{amssymb, amsmath, amscd, dcpic, pictexwd}
\usepackage{hyperref}
\usepackage{graphicx}
\usepackage{pdfpages}
\usepackage{color}
\usepackage[makeroom]{cancel}
\usepackage{mathtools}
\usepackage{setspace}
\usepackage{mathrsfs}

\allowdisplaybreaks

\newtheorem{thm}{Theorem}[section]
\newtheorem{lem}[thm]{Lemma}
\newtheorem{prop}[thm]{Proposition}
\newtheorem{claim}[thm]{Claim}
\newtheorem{cor}[thm]{Corollary}

\theoremstyle{definition}
\newtheorem{defn}[thm]{Definition}
\newtheorem{exm}[thm]{Example}

\theoremstyle{remark}
\newtheorem{rem}[thm]{Remark}

\numberwithin{equation}{section}

\DeclareMathOperator*{\pr}{pr}

\DeclareMathOperator*{\Aut}{Aut}

\DeclareMathOperator*{\rk}{rk}
\DeclareMathOperator*{\Der}{Der}

\DeclareMathOperator*{\Spec}{Spec}

\DeclareMathOperator*{\End}{End}

\DeclareMathOperator*{\im}{Im}

\newcommand{\lemma}[2]{\begin{lem}\label{#1}#2\end{lem}}

\newcommand{\eq}[2]{\begin{equation}\label{#1}#2 \end{equation}}

\newcommand{\ea}[1]{\begin{eqnarray*}#1\end{eqnarray*}}

\newcommand{\al}{\alpha}
\newcommand{\be}{\beta}

\newcommand{\pa}{\partial}

\newcommand{\cH}{\mathcal H}

\newcommand{\cJ}{\mathcal J}

\newcommand{\cM}{\mathcal M}
\newcommand{\cN}{\mathcal N}
\newcommand{\cO}{\mathcal O}

\newcommand{\cS}{\mathcal S}

\newcommand{\cY}{\mathcal Y }

\newcommand{\C}{\mathbb C}

\newcommand{\F}{\mathbb F }

\newcommand{\HH}{\mathbb{H}}
\renewcommand{\H}{\HH}

\newcommand{\N}{\mathbb N}

\newcommand{\PP}{\mathbb P}
\renewcommand{\P}{\PP}
\newcommand{\Q}{\mathbb Q}
\newcommand{\R}{\mathbb R}

\newcommand{\Z}{\mathbb Z}

\newcommand{\fg}{\mathfrak{g}}
\renewcommand{\gg}{\fg}

\newcommand{\gl}{\mathfrak{l}}
\newcommand{\gm}{\mathfrak{m}}

\newcommand{\gs}{\mathfrak{s}}

\newcommand{\into}{\hookrightarrow}

\newcommand{\ra}{\rightarrow}
\newcommand{\half}{{1\over 2}}
\newcommand{\bra}{{\langle}}
\newcommand{\ket}{{\rangle}}

\newcommand{\bs}{\bigskip}

\newcommand{\kerr}{{\mbox{ker~}}}
\renewcommand{\ker}{\kerr}

\newcommand{\sol}{{\mbox{sol}}}
\newcommand{\bl}{\hskip.2in}

\newcommand{\Cx}{\mathbb{C}^\times}
\DeclareMathOperator*{\re}{Re}

\newcommand{\xra}{\xrightarrow}

\newcommand{\quash}[1]{}

\def\comment#1{{}}
\def\question#1{{}}

\begin{document}

\title[\resizebox{5.8in}{!}{Differential zeros of period integrals and generalized hypergeometric functions}]{Differential zeros of period integrals and generalized hypergeometric functions}
\author{Jingyue Chen, An Huang, Bong H. Lian, Shing-Tung Yau}

\date{}

\dedicatory{}


\begin{abstract}
In this paper, we study the zero loci of locally constant sheaves of the form $\delta\Pi$, where $\Pi$ is the period sheaf of the universal family of CY hypersurfaces 
in a suitable ambient space $X$, and $\delta$ is a given differential operator on the space of sections $V^\vee=\Gamma(X,K_X^{-1})$. Using earlier results of three of the authors and their collaborators, we give several different descriptions of the zero locus of $\delta\Pi$. As applications, we prove that the locus is algebraic and in some cases, non-empty. We also give an explicit way to compute the polynomial defining equations of the locus in some cases. This description gives rise to a natural stratification to the zero locus.
\end{abstract}

\maketitle

\begin{spacing}{}\parskip=0cm \tableofcontents \end{spacing} 


\section{Introduction}

Zeros of special functions have been of interests to many authors since the times of Riemann. He of course famously conjectured that the zeros of the Riemann zeta functions occur only on a certain critical line. Inspired by works of Stieltjes, Hilbert and Klein, Hurwitz \cite{Hu1}\cite{Hu2} and Van Vleck \cite{V} determined the number of zeros of the Gauss hypergeometric function ${}_2F_1(a,b,c;z)$ for real $a,b,c$. Subsequently, many authors generalized their results to confluent hypergeometric functions. Runckel \cite{R} gave a simpler proof of the results of Hurwitz and Van Vleck using the argument principle.
Eichler and Zagier \cite{EZ} gave a complete description of the zeros of the Weierstrass $\wp$ function in terms of a classical Eisenstein series. Duke and Imamo\={g}lu \cite{DI} later used it to prove transcendence of values of certain classical generalized hypergeometric functions at algebraic arguments. 
More recently following Hille \cite{H}, Ki and Kim \cite{KiK} studied the zeros of generalized hypergeometric functions of the form ${}_pF_p$. For real parameters for such a function, they showed that it can only have finitely many zeros, and that they are all real. 

Since all (except the Riemann zeta function) of those special functions are solutions to ordinary differential equations, it is natural to consider the higher dimensional analogues of these functions and their zeros. 
It is well known that the theory of Gel'fand-Kapranov-Zelevinsky (GKZ) hypergeometric functions \cite{GKZ} generalize 
classical {\it special functions}, including the Euler-Gauss, Appell, Clausen-Thomae, Lauricella hypergeometric functions, and
their multivariable generalizations. Therefore, GKZ hypergeometric functions can be viewed as generalized special functions.
Since the theory of tautological systems generalizes the GKZ theory \cite{LY}, solutions to tautological systems 
and their derivatives can be thought of as further generalizations of special functions. 
The zero loci of their derivatives amount to zeros of these vast generalizations of those
for classical special functions.

In this paper, we shall study the zeros of derivatives of GKZ hypergeometric functions and their generalizations in the context of Calabi-Yau geometry. It is well-known that period integrals of CY hypersurfaces in a toric variety are GKZ hypergeometric functions. Moreover, since these functions are local sections of locally constant sheaves, each admits a multi-valued analytic continuation. Thus it is natural to consider zero loci that are monodromy invariant. Recall that the period sheaf $\Pi$ of the universal family of smooth CY hypersurfaces in a suitable ambient space $X$ form a locally constant sheaf, which is generated by pairings between a nonvanishing holomorphic top form and middle dimensional cycles on a CY hypersurface. Since every such hypersurface has at least one nonzero period, the zero locus of the period sheaf is always empty. However, as it turns out, it is more natural to consider the zero locus of a locally constant sheaf of the form $\delta\Pi$, where $\delta$ is a differential operator on the affine space $V^\vee=\Gamma(X,K_X^{-1})$. This zero locus will be the main object of study in this paper.

\bs

{\it We will follow closely the notations introduced in \cite{HLZ}\cite{BHLSY}.}
Given a Lie algebra $\hat\gg$, a $\hat\gg$-module $V^\vee$, and a $\hat\gg$-invariant ideal $I$ of the commutative algebra $\C[V]$, then a {\it tautological system} $\tau$ is a $D_{V^\vee}$-module of the form
\[\tau=D_{V^\vee}/(D_{V^\vee}\tilde{I}+D_{V^\vee}\hat\gg)\]
where $\tilde{I}\subset D_{V^\vee}$ is the Fourier transform of $I$.
In this paper, we consider the following special case of $\tau$.

Let $G$ be a connected complex algebraic group. Let $X$ be a complex projective $G$-variety and let $L$ be a very ample $G$-equivariant line bundle over $X$. This gives rise to a $G$-equivariant embedding 
\[X\ra\P(V),\]
where $V=\Gamma(X,L)^\vee$. We assume that the action of $G$ on $X$ is locally effective, i.e. $\ker(G\ra\Aut (X))$ is finite. 
Let $\hat{G}:=G\times\Cx$, whose Lie algebra is $\hat{\gg}=\gg\oplus \C e$, where $e$ acts on $V$ by identity. We denote by $Z:\hat{G}\ra \text{GL}(V)$ the group action induced on $V$, and by $Z:\hat{\gg}\ra\End(V)$ the corresponding Lie algebra representation. Note that under our assumption, $Z:\hat{\gg}\ra\End(V)$ is injective.

Let $\hat{\iota}: \hat{X}\subset V $ be the cone of $X$, and $I(\hat{X})$ its defining ideal. Let $\beta:\hat{\gg}\ra \C$ be a Lie algebra homomorphism. Then a {\it tautological system } as defined in \cite{LSY}\cite{LY} is the cyclic $D$-module on $V^\vee$
\[\tau(X,L,G,\beta)=D_{V^\vee}/\big(D_{V^\vee}\tilde{I}+D_{V^\vee}(Z(x)+\beta(x), x\in\hat{\gg})\big),\]
where $$\tilde{I}=\{\tilde{P}\mid P\in I(\hat{X})\}$$
is the ideal of the commutative subalgebra $\C[\partial]\subset D_{V^\vee}$ obtained by the Fourier transform of $I(\hat{X})$. Here $\tilde{P}$ denotes the Fourier transform of $P$.

Given a basis $\{a_1,\ldots, a_n\}$ of $V$, we have $Z(x)=\sum_{ij}x_{ij}a_i\frac{\partial}{\partial{a_j}}$, where $(x_{ij})$ is the matrix representing $x$ in the basis. Since the $a_i$ are also linear coordinates on $V^\vee$, we can view $Z(x)\in\Der \C[V^\vee]\subset D_{V^\vee}$. In particular, the identity operator $Z(e)\in\End V$ becomes the Euler vector field on $V^\vee$.

Let $X$ be an $m$-dimensional compact complex manifold such that its anti-canonical line bundle $K_X^{-1}$ is very ample. Let $L:=K_X^{-1}$. We shall regard the basis elements $a_i$ of $V=\Gamma(X,L)^{\vee}$ as linear coordinates on $V^{\vee}$. Let $B:=\Gamma(X,L)_{sm}\subset V^{\vee}$ be the space of smooth sections.
 Let $\pi:\cY\ra B$ be the family of smooth CY hyperplane sections $Y_b\subset X$, and let $\H^{\text{top}}$ be the Hodge bundle over $B$ whose fiber at $b\in B$ is the line $\Gamma(Y_b,\omega_{Y_b})\subset H^{m-1}(Y_b)$. In \cite{LY} the period integrals of this family are constructed by giving a canonical trivialization of $\H^{\text{top}}$. Let $\Pi$ be the period sheaf of this family, i.e. the locally constant sheaf generated by the period integrals.
 Let $G$ be a connected algebraic group acting on $X$.

\begin{thm}[See \cite{LY}]\label{tautological}
The period integrals of the family $\pi:\cY\ra B$ are solutions to 
\[\tau\equiv\tau(X,K_X^{-1},G,\beta_0)\]
where $\beta_0$ is the Lie algebra homomorphism with $\beta_0(\gg)=0$ and $\beta_0(e)=1$.
\end{thm}

In \cite{LSY} and \cite{LY}, it is shown that if $G$ acts on $X$ by finitely many orbits, then 
$\tau$ is regular holonomic. {\it We shall assume this holds throughout the paper.}

\bs

Let $R=\C[V]/I(\hat{X})$. Let $f=\sum a_ia_i^\vee$ be the universal section. Then the Lie algebra $\hat\gg=\fg\oplus\C e$ acts on $ R[V^\vee]e^f$ by the homomorphism
$Z^\vee:\hat\gg\ra\End V^\vee$ which is dual to Lie algebra action $Z$ on $V$. Thus it takes the form
$$Z^\vee(x)=-\sum x_{ij}a^\vee_j{\partial\over\partial a^\vee_i}-\beta(x),\bl x\in\hat\gg.$$
Here $\{a_i\}$,$\{a_i^\vee\}$ are the bases of $V,V^\vee$ dual to each other.
Note that since $I(\hat X)$ is a $\hat\gg$-invariant ideal of $\C[V]$, there is an induced $\hat\gg$-action on $R$ hence on $R[V^\vee]e^f=R[a]e^f$. 
Recall that the $D_{V^\vee}$-module structure on $R[V^\vee]e^f$ is that $a_i\in D_{V^\vee}$ acts by left multiplication, while $\partial_i\in D_{V^\vee}$ acts by the usual derivative $\partial\over\partial a_i$. In particular, this action commutes with the $\hat\gg$-action given by $Z^\vee$, and with left multiplication by $R$.

\begin{thm}\cite{BHLSY},\cite{HLZ}\label{isom}
There is a canonical isomorphism of $D_{V^\vee}$-modules
{\small\begin{align*}
\tau(X,L,G,\beta_0)&\xleftrightarrow{\Phi} R[V^\vee]e^f/
\hat{\gg}(R[V^\vee]e^f)\\ 
1\quad\quad&\longleftrightarrow \quad\quad\quad e^f. 
\end{align*}}
\end{thm}

Denote by $\sol(\tau)$ the sheaf of classical solutions to $\tau$. We will prove in Section 2
\begin{thm}\label{equiv0}
Let $\delta\in D_{V^\vee}$, and $b\in V^\vee$.
The following statements are equivalent:
\begin{enumerate}
\item \label{cond1} $\delta \gs(b)=0$ for all $\gs\in \sol(\tau)_b$.

\item  \label{cond2} $\delta e^{f(b)}=0$ in $Re^{f(b)}/\hat{\gg}Re^{f(b)}$, i.e. $\delta e^{f(b)}\in \hat{\gg}Re^{f(b)}$.
\end{enumerate}
\end{thm}
This theorem generalizes \cite[Corollary 4.2]{CHL}.

For any $\delta\in D_{V^\vee}$, we introduce
\begin{equation}\label{Np}
\cN(\delta)=\{b\in B\mid \delta \gs(b)=0,\, \forall \gs\in \sol(\tau)_b\}.
\end{equation}
This will be a main object of study in this paper. By Theorem \ref{equiv0}, we have
\[\cN(\delta)=\{b\in B\mid \delta e^{f(b)}\in \hat{\gg}(Re^{f(b)})\}.\]
In the special case $\delta=p(\partial)\in\C[\partial]$ has constant coefficients, we have
$$p(\partial)e^{f(b)}=p(a^\vee)e^{f(b)}$$
Thus making the identification $R\equiv\tilde R$ by Fourier transform ~~$\tilde{}:R\ra\tilde R$, $p(a^\vee)\mapsto p(\partial)$, we get 
\[\cN(p)\equiv\cN(\tilde p)=\{b\in B\mid p(a^\vee) e^{f(b)}\in \hat{\gg}(Re^{f(b)})\}.\]
This recovers the definition of $\cN(p)$ introduced in \cite{CHL}.

We will prove in Section 5
\begin{thm}
If $\delta\in D_{V^\vee}$ is homogeneous under scaling by $\C^\times$, $\cN(\delta)$ is algebraic.
\end{thm}

In Sections 6 and 8, we discuss the non-emptiness of $\cN(\delta)$ in a number of cases. In Section 7, we give an explicit way to compute the polynomial equations defining $\cN(\delta)$ in $\P V^\vee$ in the case $X=\P^m$. We also show that $\cN(\delta)$ has a natural stratification in this case.

\bs
\textbf{Acknowledgements.} We would like to thank Masaki Kashiwara for helpful discussions. We also thank Mei-Heng Yueh for helping us with computer calculations. We are grateful to the referees for helpful suggestions and corrections, all of which have now been incorporated into the paper. Research of J.C. is partially supported by a Special Financial Grant from the China Postdoctoral Science Foundation 2016T90080. Part of the work was done during her visit to Brandeis University. B.H.L is partially supported by an NSF FRG grant MS-1564405 and a Simons Collaboration Grant on HMS and Application 2015.

\section{A coinvariant description of differential zeros}

Let $J=D_{V^\vee}\tilde{I}+D_{V^\vee}(Z(x)+\beta_0(x), x\in\hat{\gg})$ be the defining left ideal of a regular holonomic tautological system $\tau$. Then $\tau=D_{V^\vee}/J$.
Since $\tau$ is cyclic, $\sol(\tau)$ can be identified as a subsheaf of local analytic functions in $\cO_{V^\vee}\equiv\cO^{an}_{V^\vee}$ annihilated by the left ideal $J$. Then we have the canonical isomorphism of sheaves
$$
\cH om{}_{D_{V^\vee}}(\tau,\cO_{V^\vee})\to \sol(\tau),\bl \varphi\mapsto \varphi(1).
$$

\begin{thm}\label{equiv}
Let $\delta\in D_{V^\vee}\equiv\C[a_i,\partial_{a_i}]$, and $b\equiv\sum b_i a^\vee_i\in V^\vee$.
The following statements are equivalent:
\begin{enumerate}
\item \label{cond1} $\delta \gs(b)=0$ for all $\gs\in \sol(\tau)_b$,

\item  \label{cond2} $\delta e^{f(b)}=0$ in $Re^{f(b)}/\hat{\gg}Re^{f(b)}$, i.e. $\delta e^{f(b)}\in \hat{\gg}Re^{f(b)}$.

\item \label{2cond22} $\delta\in \gm_bD_{V^\vee}+J $, where $\gm_b:=\langle a_i-b_i\rangle$ is the ideal sheaf of the point $b$.
\end{enumerate}
\end{thm}

\begin{proof}First we prove \eqref{cond1}$\Leftrightarrow$\eqref{cond2}.
Consider the evaluation map
$$e_b: D_{V^\vee,b}\ra \oplus_\alpha\C\partial^\alpha\equiv \C[\partial], \bl \sum g_\alpha\partial^\alpha\mapsto \sum g_\alpha(b)\partial^\alpha.$$
Let $i_b: b\ra V^\vee$ be the inclusion and $\cO_b\equiv \C$ be the constant sheaf over $b$.
\begin{claim}\label{iso}
The morphism
\[e'_b: i^*_b D_{V^\vee}=\cO_b\otimes_{i_b^{-1}\cO_{V^\vee}}D_{V^\vee,b}\xra{\simeq} e_b(D_{V^\vee,b})=\C[\partial],\; 1\otimes \delta\mapsto e_b(\delta)\]
is well-defined and it is an isomorphism.
\end{claim}
\begin{proof}
It is clear that the map $$e'_b: \cO_b\otimes_\C D_{V^\vee,b}\ra e_b(D_{V^\vee,b}),\bl 1\otimes \delta\mapsto e_b(\delta)$$
is well-defined. Let $f\in i_b^{-1}\cO_{V^\vee}=\cO_{V^\vee,b}$, then
 \[1\otimes f\delta-f(b)\otimes\delta\mapsto e_b(f\delta)-f(b)e_b(\delta)=0.\]
Thus $e'_b$ descends and it is well-defined on $i^*_b D_{V^\vee}$.

Surjectivity: For any $\delta_c:=\sum_\alpha c_\alpha\partial^\alpha\in \C[\partial]$ where $c_\alpha\in\C$, we have $\delta_c\in D_{V^\vee}$ and $e_b(\delta_c)=\delta_c$. Thus  $e'_b(1\otimes \delta_c)=e_b(\delta_c)=\delta_c.$

Injectivity: Let $\gm_b:=\langle a_i-b_i\rangle$ be the ideal sheaf of the point $b$. Then $e'_b(1\otimes \sum_\alpha g_\alpha\partial^\alpha)=\sum g_\alpha(b)\partial^\alpha=0$ implies that $g_\alpha\in \gm_b$ for all $\alpha$. Thus $1\otimes \sum_\alpha g_\alpha\partial^\alpha=\sum_\alpha g_\alpha(b)\otimes\partial^\alpha=0$, which means that $\ker e'_b=0$.
\end{proof}

Since $\cO_b\otimes_{i_b^{-1}\cO_{V^\vee}}J_{b}=\{1\otimes\delta\mid\delta\in J_b\}$,
similarly we can show that $$\cO_b\otimes_{i_b^{-1}\cO_{V^\vee}}J_{b}\simeq e_b(J_{b}).$$

Next we claim that $e_b$ induces a map $e_b: \tau_b\ra i_b^*\tau.$
Since $i_b^{-1}\tau=\tau_b$, we have 
$$i_b^*\tau:=\cO_b\otimes_{i_b^{-1}\cO_{V^\vee}}i_b^{-1}\tau=\cO_b\otimes_{i_b^{-1}\cO_{V^\vee}}\tau_b.$$ 
Consider the exact sequence 
\[0\ra J_b\xra{\iota} D_{V^\vee,b}\xra{p}\tau_b=D_{V^\vee,b}/J_b\ra 0.\]
Since tensoring over any ring is right exact, we have
\[\cO_b\otimes_{i_b^{-1}\cO_{V^\vee}}J_b\xra{\cO_b \otimes \iota} \cO_b\otimes_{i_b^{-1}\cO_{V^\vee}}D_{V^\vee,b}\xra{\cO_b\otimes p}\cO_b\otimes_{i_b^{-1}\cO_{V^\vee}}(D_{V^\vee,b}/J_b)\ra 0.\]
Thus 
\[\ker \cO_b\otimes p=\im\cO_b\otimes \iota=\cO_b\otimes_{i_b^{-1}\cO_{V^\vee}} J_b,\]
hence 
\[i_b^*\tau=\cO_b\otimes_{i_b^{-1}\cO_{V^\vee}}(D_{V^\vee,b}/J_b)\simeq (\cO_b\otimes_{i_b^{-1}\cO_{V^\vee}}D_{V^\vee,b})/(\cO_b\otimes_{i_b^{-1}\cO_{V^\vee}}J_b).\]
Therefore by Claim \ref{iso} we have 
\[i_b^*\tau\simeq(\cO_b\otimes_{i_b^{-1}\cO_{V^\vee}}D_{V^\vee,b})/(\cO_b\otimes_{i_b^{-1}\cO_{V^\vee}}J_b)\simeq e_b(D_{V^\vee,b})/e_b(J_b).\]
Now we have a surjective map
\[e_b:\tau_b\ra e_b(D_{V^\vee,b})/e_b(J_b)\simeq i_b^*\tau.\]
Consider the pairing 
\begin{equation}\label{pairing00}
\tau\otimes_\C\cH om{}_{D_{V^\vee}}(\tau,\cO_{V^\vee})\ra \cO_{V^\vee},\bl \delta\otimes\varphi\mapsto\delta(\varphi).
\end{equation}
And note that evaluation is $\cO_{V^\vee}$-bilinear. Taking a $b$-germ of \eqref{pairing00} yields
\[\tau_b\otimes_\C\cH om{}_{D_{V^\vee}}(\tau,\cO_{V^\vee})_b\ra \cO_{V^\vee,b}.\]
Applying $\cO_b\otimes_{i_b^{-1}\cO_{V^\vee}}-$ to both sides,  we get
\begin{equation}\label{tensor}
\alpha: \cO_b\otimes_{i_b^{-1}\cO_{V^\vee}}\tau_b\otimes_\C\cH om{}_{D_{V^\vee}}(\tau,\cO_{V^\vee})_b\ra\cO_b\otimes_{i_b^{-1}\cO_{V^\vee}}\cO_{V^\vee,b}.
\end{equation}
The morphism is given by $\alpha(1\otimes \delta\otimes \varphi)= 1\otimes \varphi(\delta)=e_b(\varphi(\delta))$.
Here 
$$\cO_b\otimes_{i_b^{-1}\cO_{V^\vee}}\cO_{V^\vee,b}=i_b^*(\cO_{V^\vee})=\cO_b.$$ 
Since $\tau$ is regular holonomic, it follows that 
\[\cH om{}_{D_{V^\vee}}(\tau,\cO_{V^\vee})_b\xra{\simeq} \cH om{}_\C(i_b^*\tau,\cO_b),\]
where $\varphi\mapsto \bar{\varphi}$ and $\bar{\varphi}(e_b(\delta)):=e_b(\varphi(\delta))$.

Next, consider the canonical non-degenerate pairing
\begin{equation}\label{pairing1}
\beta: i_b^*\tau\otimes\cH om{}_\C(i_b^*\tau,\cO_b)\ra \cO_b\equiv \C,
\end{equation}
together with pairing \eqref{tensor} we have a diagram
\[\xymatrix{
&\cO_b\otimes_{i_b^{-1}\cO_{V^\vee}}\tau_b\otimes_\C\cH om_{D_{V^\vee}}(\tau,\cO_{V^\vee})_b\ar[d]_{\simeq}^\gamma\ar[r]^{\quad\quad\quad\alpha} &\cO_b\otimes_{i_b^{-1}\cO_{V^\vee}}\cO_{V^\vee,b}\ar@{=}[d]\\
&i_b^*\tau\otimes\cH om_\C(i_b^*\tau,\cO_b)\ar[r]^{\beta} &\cO_b.
}\]
Since $$\beta\circ\gamma((1\otimes \delta)\otimes \varphi)=\beta(e_b(\delta)\otimes \bar{\varphi})
=\bar{\varphi}(e_b(\delta))=e_b(\varphi(\delta))=\alpha(1\otimes \delta\otimes \varphi),$$
 the above diagram commutes.

Since 
\[\cH om{}_{D_{V^\vee}}(\tau,\cO_{V^\vee})\xrightarrow{\simeq}\sol(\tau),\quad \varphi\mapsto \varphi(1),\]
then condition \eqref{cond1}: $\delta \gs(b)=0$ for all $\gs\in \sol(\tau)_b$ is equivalent to 
$$(\delta \varphi(1))(b)=(\varphi(\delta\cdot 1))(b)=e_b(\varphi(\delta))=\beta(e_b(\delta)\otimes \bar{\varphi})=0$$ for all $\varphi\in \cH om_{D_{V^\vee}}(\tau,\cO_{V^\vee}).$
By the non-degeneracy of pairing \eqref{pairing1}, this is equivalent to $e_b(\delta)=0$ in $i^*_b\tau.$

On the other hand, by the isomorphism $$\tau\xrightarrow{\simeq} (R[V^\vee]e^f/\hat{\gg}R[V^\vee]e^f),\; \delta\mapsto \delta e^f,$$
we have
$$i_b^*\tau\xrightarrow{\simeq} i_b^*(R[V^\vee]e^f/\hat{\gg}R[V^\vee]e^f)\simeq i_b^*R[V^\vee]e^f/\hat{\gg}i_b^*R[V^\vee]e^f\simeq Re^{f(b)}/\hat{\gg}Re^{f(b)}.$$
Thus  $e_b(\delta)=0$ in $i_b^*\tau$ is equivalent to $(\delta e^f)(b)=0$ in $Re^{f(b)}/\hat{\gg}Re^{f(b)}$, which is the condition \eqref{cond2}.
This completes the proof of \eqref{cond1}$\Leftrightarrow$\eqref{cond2}.

Next, we prove \eqref{cond2}$\Leftrightarrow$\eqref{2cond22}.

\begin{claim}
The following diagram commutes:
\[\xymatrix{
&\tau\ar[d]_{\Phi}\ar[r]^{e_b} &i^*_b\tau\ar[d]^{\cO_b\otimes\Phi}\\
&R[V^\vee]e^f/\hat{\gg}R[V^\vee]e^f\ar[r]^{e_b} &Re^{f(b)}/\hat{\gg}Re^{f(b)}
}\]
where the $e_b$ are evaluation maps, and 
{\small\ea{\Phi(\sum g_\al\pa^\al)&=&(\sum g_\al\pa^\al )\cdot e^f=\sum g_\al(a^\vee)^\al e^f\\ 
(\cO_b\otimes\Phi)(\sum g(b)_\al\pa^\al)&=&\left((\sum g(b)_\al\pa^\al)\cdot e^f\right)(b)=(\sum g(b)_\al (a^\vee)^\al e^f)(b)\\
&=&\sum g(b)_\al (a^\vee)^\al e^{f(b)}.}}
\end{claim}

Define the map
\[\Theta_b: D_{V^\vee}\ra Re^{f(b)}/\hat{\gg}Re^{f(b)},\quad \delta\mapsto \delta e^{f(b)}.\]
Let \[\bar{\Theta}_b: \tau=D_{V^\vee}/J\ra Re^{f(b)}/\hat{\gg}Re^{f(b)},\quad \delta\mapsto \delta e^{f(b)}.\]
Let \[\Theta'_b: \tau\ra \tau/\gm_b\tau\xra{\simeq} i^*_b\tau,\quad \delta\mapsto \delta+\gm_b\tau\mapsto \delta(b),\]
which is an   $\cO_{V^\vee}$-module morphism.
Then we have a diagram:
\[\xymatrix{
&\tau\ar@{=}[d]\ar[r]^{\bar{\Theta}_b\quad\quad}&Re^{f(b)}/\hat{\gg}Re^{f(b)}\\
&\tau\ar[r]^{\Theta'_b\quad}&\tau/\gm_b\tau\simeq i^*_b\tau.\ar[u]^{\simeq}_{\cO_b\otimes\Phi\equiv\Phi'}
}\]

\begin{claim}
$\bar{\Theta}_b=\Phi'\circ \Theta'_b$, i.e. the above diagram commutes.
\end{claim}

\begin{proof}
Let $\delta=\sum g_\al\pa^\al,\, \bar{\Theta}_b(\delta)=(\delta e^f)(b)=\sum g(b)_\al(\pa^\al e^f)(b).$ 
On the other hand, $\Phi'(\delta(b))=(\delta(b)e^f)(b)=\sum (g(b)_\al\pa^\al e^f)(b)=\bar{\Theta}_b(\delta).$
Thus $\bar{\Theta}_b=\Phi'\circ \Theta'_b$.
\end{proof}

Let $\pr: D_{V^\vee}\ra\tau=D_{V^\vee}/J$ be the projection.

\begin{prop}
For all $b\in V^\vee,$ $$\ker \Theta_b=\gm_bD_{V^\vee}+J.$$
(Note that $\gm_bD_{V^\vee}$ is a right ideal and $J$ is a left ideal.)
\end{prop}

\begin{proof}
By the previous claim $\Theta_b=\bar{\Theta}_b\circ pr=\Phi'\circ \Theta'_b\circ \pr.$
Since $\Phi'$ is an isomorphism,
\[\ker\Theta_b=\ker \Theta_b'\circ \pr.\]
$\ker \Theta_b'\circ \pr=\ker(D_{V^\vee}\ra \tau=D_{V^\vee}/J\ra \tau/\gm_b\tau)=\gm_bD_{V^\vee}+J.$
\end{proof}
Therefore given $\delta\in D_{V^\vee}$, then $\delta e^{f(b)}=0$ in $Re^{f(b)}/\hat{\gg}Re^{f(b)}$ iff $\Theta_b(\delta)=0$ iff $\delta\in \ker \Theta_b=\gm_bD_{V^\vee}+J$, i.e. \eqref{cond2}$\Leftrightarrow$\eqref{2cond22}. This completes the proof of Theorem \ref{equiv}.
\end{proof}

The theorem shows that for each $b\in V^\vee$, the membership condition $\delta e^{f(b)}\in\hat\fg Re^{f(b)}$ determines exactly if $b$ is a zero of the sheaf $\delta\sol(\tau)$ of analytic functions. Thus describing the vector subspace $\hat\fg Re^{f(b)}\subset R e^{f(b)}$ is crucial in understanding differential zeros of the solutions to $\tau$ in general, and of generalized hypergeometric functions in particular. In Appendix \ref{appendix A}, we give an explicit basis for $\hat\fg Re^{f(b)}$ for a number of interesting examples.

\section{Analyticity along singularity}

In this section, we shall consider the zero locus of certain sheaf of analytic functions on a complex manifold $B$.\footnote{We thank Professor M. Kashiwara for his helpful insights which provide the basis for the analytic argument in this section.}

\begin{defn}
Let $B$ be a complex manifold. A locally constant sheaf $\cS$ of finite dimensional vector spaces on $B$ is called analytic (ALCS) if it is equipped with an embedding $\cS\into \cO_B$ of sheaves. We shall identify an ALCS $\cS$ with its image in $\cO_B$ via the given embedding, and treat $\cS$ as a subsheaf of $\cO_B$.
\end{defn}

The classical solution sheaf $\sol(\tau)$ of a holonomic $D$-module $\tau$ on $B$ is an ALCS. For a given ALCS $\cS$  and for any $\delta\in D_{V^\vee}$, let $\delta \cS$ be the sheaf such that $(\delta \cS)_b=\{\delta \gs\mid \gs\in\cS_b\}$, then it is also an ALCS. An ALCS of the form $\delta \sol(\tau)$ for a tautological system $\tau$ will be our primary focus here.

\begin{defn}
Let $\overline B$ be a smooth partial compactification of $B$ such that $D=\overline B\backslash B$ is a normal crossing divisor in $\overline{B}$. We say that an ALCS $\cS$ on $B$ has regular singularity along $D$,  if for each $b_0\in D$, there exists local coordinates $z=(z_1,\ldots, z_n)$ on $\overline B$ in some polydisk $U$ centered at $b_0$ such that $U\cap D=U\cap (\bigcup^r_{i=1}\{z_i=0\})$ for some $1\leq r\leq n$ and
every $\gs\in\cS(U\backslash D)$ has the form
\begin{equation}\label{regular singular}
\gs=\sum_{\al\in\Lambda}\sum_{I\in\Theta} g_{\al,I}(z)[z]_r^{\al}[\log z]_r^{I}
\end{equation}
on $U\backslash D$, where $\Lambda$ is a finite subset of $\C^r$, $[z]_r^{\al}=z_1^{\al^1}\cdots z_r^{\al^r}$; $\Theta$ is a finite subset of $\Z_{\geq 0}^r$, $[\log z]_r^{I}=(\log z_1)^{I^1}\cdots(\log z_r)^{I^r}$, and $g_{\al,I}$ are meromorphic functions with poles along $D$.
\end{defn}

Note that if $\cS$ is the solution sheaf of a regular holonomic $D$-module with singular hypersurface being a normal crossing divisor $D$, then $\cS$ is an ALCS with regular singularity along $D$ (cf. \cite[p.862]{KK}, \cite[p.83]{SST}).

The typical situation we shall consider is when $\cS = \delta \sol(\tau)$, where $\tau$ is a regular holonomic tautological system defined on $V^\vee$ as before and $\delta\in D_{V^\vee}$. Since $B$ is a Zariski open subset of $V^\vee$, $V^\vee$ can be viewed as a smooth partial compactification of $B$.
However, it may be the case that the divisor $D=V^\vee\backslash B$ fails to be normal crossing. In that case we can remedy this by blowing up $V^\vee$ along $D$ to achieve normal crossing, which we will talk about in the next section.

Our main goal here is to show that the differential zero locus $\cN(\delta)$ of $\delta \sol(\tau)$ has an analytic closure in $V^\vee$ if $D$ is a normal crossing divisor. Then we can use the proper mapping theorem to conclude for the general case.

{\it For the rest of this section, $\cS$ is assumed to be an ALCS on $B$ with regular singularity along $D=\overline B\backslash B$.}

\subsection{Regular singularities}
For fixed $I\in\Theta$, we can combine  terms in \eqref{regular singular} with $\log$ component being $[\log z]^I_r$. Then we have a finite sum of the form $(\sum_{\al\in \Lambda} g_{\al,I}(z)[z]_r^\al)[\log z]^I_r$. Let $\al_{I1},\ldots,\al_{I \Lambda_I}$ denote all the $\al$'s that appear in this sum, and let $g_{Ik}(z):=g_{\al_{Ik},I}(z)$. Then we can rewrite \eqref{regular singular} as 
\begin{equation}\label{regular singular II}
\gs=\sum_{I\in\Theta}\big(\sum_ {k=1}^{\Lambda_I}g_{Ik}(z)[z]_r^{\al_{Ik}}\big)[\log z]_r^{I}.
\end{equation}
For fixed $I$, if there exist $k,k'$ such that $\al_{Ik}-\al_{Ik'}=n_{I}\in\Z^r$, then 
\[g_{Ik}(z)[z]_r^{\al_{Ik}}+g_{Ik'}(z)[z]_r^{\al_{Ik'}}=(g_{Ik}(z)+g_{Ik'}(z)[z]_r^{n_I})[z]_r^{\al_{Ik}}\]
and $g_{Ik}(z)+g_{Ik'}(z)[z]_r^{n_I}$ is a meromorphic function with poles along $\bigcup^r_{i=1}\{z_i=0\}$. So without loss of generality we can assume further in the expression \eqref{regular singular II} that for each $I$,
\begin{equation}\label{reduced}
\forall 1\leq k\leq \Lambda_I ,\,\re\al_{Ik}\in [0,1)^r\text{ and }\forall \,1\leq k\neq k'\leq \Lambda_I ,\, \al_{Ik}\neq\al_{Ik'}.
\end{equation}
We say that  $\gs$ is of {\it reduced form} if \eqref{reduced} holds.

\begin{prop}\label{general}
Assume $\cS$ on $B$ has regular singularity along $D$. For $b_0\in D$, let $U$ be a polydisk  centered at $b_0\in U\cap D=U\cap (\bigcup^r_{i=1}\{z_i=0\})$ such that for every $\gs\in \cS(U\backslash D)$,
\[\gs=\sum_{I\in\Theta^{(\gs)}}\big(\sum_ {k=1}^{\Lambda_I^{(\gs)}}g_{Ik}^{(\gs)}(z)[z]_r^{\al_{Ik}^{(\gs)}}\big)[\log z]_r^{I} 
\]
on $U\backslash D$ and is of reduced form.
Then $\gs(b)=0$ for all $\gs\in \cS_b$ if and only if $g_{Ik}^{(\gs)}(z(b))=0$ for all $ g_{Ik}^{(\gs)}$ on $U\backslash D$.
\end{prop}
We are going to prove this proposition for $r=1$ and $r=2$. Then by a straightforward induction the proposition holds for general cases.

\subsection{Case $r=1$}
Consider  $\gs\in \cS(U\backslash D)$,
\begin{equation}\label{expression}
\gs=\sum_{j=0}^d (\sum_{k=1}^{\Lambda_j} g_{jk}(z)z_1^{\al_{jk}})(\log z_1)^j
\end{equation}
where $g_{jk}(z)$ are meromorphic functions with poles along $\{z_1=0\}$.
Then $\gs$ is of reduced form if it satisfies further that
\begin{equation}\label{assump}
\re\al_{jk}\in [0,1)\text{ and when }k\neq k',\, \al_{jk}\neq\al_{jk'}.
\end{equation}

Suppose for some $b\in U\backslash D$, $\gs(b)=0$ for all $\gs\in \cS_b$. Then the zero locus is monodromy invariant. Let $z(b)$ denote the coordinate of $b$ in $U$, then $z_1(b)\neq 0$. Fix $z_i=z_i(b)$ for $2\leq i\leq n$ in $\gs$, the analytic continuation of $\gs$ around $z_1=0$ also vanishes at $b$.
Let $\log z_1(b)=w+2\pi im,\,m\in\Z$ for some $w\in\C$,  then
\[0=\gs(m)=\sum_{j=0}^d (\sum_{k=1}^{\Lambda_j}  c_{jk}e^{2\pi im\al_{jk}})(w+2\pi im)^j,\quad\forall m\in \Z\] 
where $c_{jk}=g_{jk}(z(b))e^{\al_{jk}w}\in\C.$

\begin{claim}
$c_{jk}=0$ for all $0\leq j\leq d,1\leq k\leq \Lambda_j$.
\end{claim}
\begin{proof}
Let $\{\al_1,\ldots,\al_s\}:=\{\al_{jk}\}_{j,k}$ where $\al_1,\ldots,\al_s$ are pairwise distinct. Then we can write 
\[\gs(m)=\sum_{l=1}^s e^{2\pi im\al_{l}}(\sum_{\{j,k\mid\al_{jk}=\al_l\}} c_{jk}(w+2\pi im)^j).\]
Let $P'_l(m):=\sum_{\{j,k\mid\al_{jk}=\al_l\}}  c_{jk}(w+2\pi im)^j$. Since \eqref{assump} holds, the $j$'s appearing in the summands are pairwise distinct. We have
\begin{equation}\label{sum 1}
0=\gs(m)=\sum_{l=1}^s e^{2\pi im\al_{l}}P'_l(m),\quad\forall m\in \Z.
\end{equation}

Let $\beta:=\min_{1\leq l\leq s}\{\im \al_{l}\}$ and let $p$ be the number of $\al_l$'s that reaches this minimum. Without loss of generality we can assume $\im \al_{1}=\cdots=\im \al_{p}=\beta.$
 Consider
{\small\begin{align*}
0=&\left(\sum_{l=1}^s e^{2\pi im\al_{l}}P_{l}'(m)\right)/e^{2\pi im(i\beta)}\\
=&e^{2\pi im\re\al_{1}}P'_{1}(m)+\cdots+e^{2\pi im\re\al_{p}}P'_{p}(m)+\sum_{l=p+1}^s e^{2\pi m(\beta-\im\al_{l})}e^{2\pi im\re\al_{l}}P_{l}'(m).
\end{align*}}
Since $\beta-\im\al_{l}<0$ for $l>p$, let $m\ra\infty$,
\[\lim_{m\ra\infty} |e^{2\pi m(\beta-\im\al_{l})}e^{2\pi im\re\al_{l}}P_{l}'(m)|=\lim_{m\ra\infty} |e^{2\pi m(\beta-\im\al_{l})}P_{l}'(m)|=0\,\text{ for }l>p.\]
Thus 
\begin{equation}\label{limsum}
\lim_{m\ra\infty} e^{2\pi im\re\al_{1}}P'_{1}(m)+\cdots+e^{2\pi im\re\al_{p}}P'_{p}(m)=0.
\end{equation}
We have
\begin{align*}
\sum_{l=1}^p e^{2\pi im\re\al_{l}}P'_{l}(m)&=\sum_{l=1}^p e^{2\pi im\re\al_{l}}\big(\sum_{\{j,k\mid\al_{jk}=\al_l\}} c_{jk}(w+2\pi im)^j\big)\\
&=\sum_{j=0}^d (w+2\pi im)^j\big(\sum_{l=1}^p  \sum _{\{1\leq k\leq \Lambda_j\mid\al_{jk}=\al_l\}}c_{jk} e^{2\pi im\re\al_{jk}}   \big).
\end{align*}

Since for every $0\leq j\leq d$, $\sum_{l=1}^p  \sum _{\{k\mid\al_{jk}=\al_l\}}c_{jk} e^{2\pi im\re\al_{jk}}   $ is bounded for all $m$, then \eqref{limsum} implies
\begin{equation}\label{sum 2}
\lim_{m\ra\infty} \sum_{l=1}^p  \sum _{\{1\leq k\leq \Lambda_j\mid\al_{jk}=\al_l\}}c_{jk} e^{2\pi im\re\al_{jk}} =0 \text{ for } 0\leq j\leq d.
\end{equation}
Note that \eqref{assump} implies that for fixed $j$ and $l$, there is at most one $k$ such that $\al_{jk}=\al_l$. Thus for fixed $j$, $\al_{jk}$ appearing in \eqref{sum 2} are pairwise distinct. 
By our assumption their imaginary parts all equal $\beta$, then $\re \al_{jk}\in[0,1)$ and are pairwise distinct in the summands of \eqref{sum 2}.

\begin{lem}\label{constant}
Given $\al_l\in\R,\,a_l\in\C,\, 1\leq l\leq p$. If $\al_i-\al_j\notin\Z$ when $ i\neq j$, then
\[\lim_{m\ra\infty} e^{2\pi im\al_{1}}a_1+\cdots+e^{2\pi im\al_{p}}a_p=0\]
implies that $a_l=0$ for all $1\leq l\leq p.$
\end{lem}

\begin{proof}
When $p=1$, we have
\[\lim_{m\ra\infty} e^{2\pi im\al_{1}}a_1=0.\]
Then
\[\lim_{m\ra\infty} |a_1|=0\]
and thus $a_1=0.$
Assume that lemma holds for $p=n$. Now we consider
\begin{equation}\label{induction}
\lim_{m\ra\infty} e^{2\pi im\al_{1}}a_1+\cdots+e^{2\pi im\al_{n+1}}a_{n+1}=0.
\end{equation}
The difference of replacing $m$ by $m+1$ in \eqref{induction} and multiplying \eqref{induction} by $e^{2\pi i\al_{n+1}}$ becomes
\[\lim_{m\ra\infty} e^{2\pi im\al_{1}}( e^{2\pi i\al_{1}}-e^{2\pi i\al_{n+1}})a_1+\cdots+e^{2\pi im\al_{n}}(e^{2\pi i\al_{n}}-e^{2\pi i\al_{n+1}})a_{n}=0.
\]
Then by our inductive hypothesis we can conclude that 
\[( e^{2\pi i\al_{l}}-e^{2\pi i\al_{n+1}})a_l=0\] for $1\leq l\leq n.$ Since by our assumption $ e^{2\pi i\al_{l}}-e^{2\pi i\al_{n+1}}\neq 0$  for $1\leq l\leq n,$ then $a_1=\cdots=a_n=0.$ Thus
\[\lim_{m\ra\infty} e^{2\pi im\al_{n+1}}a_{n+1}=0\] and therefore $a_{n+1}=0$.
By induction the lemma holds for all $p$.
\end{proof}

Hence by Lemma \ref{constant} we can conclude that $c_{jk}=0$ for all $\{j,k\mid\al_{jk}=\al_l,l=1,\ldots,p\}.$ 

Now our original summation \eqref{sum 1} reduces to 
$$\sum_{l=p+1}^s e^{2\pi im\al_{l}}P_{l}'(m)=0.$$
We can repeat our strategy of considering terms that reach minimum imaginary part in this sum, then eventually we have $c_{jk}=0$ for all $j,k$.
\end{proof}
Since $c_{jk}=g_{jk}(z(b))e^{\al_{jk}w}$,  it implies $g_{jk}(z(b))=0$ for all $j,k$.

We just showed that  if $\gs(b)=0$ for all $\gs\in \cS_b$, then $g_{jk}^{(\gs)}(z(b))=0$ for all $g_{jk}^{(\gs)}.$
On the other hand, it is clear that if $g_{jk}^{(\gs)}(z(b))=0$, then $\gs(b)=0$.
Therefore Proposition \ref{general} holds if $r=1$.

\subsection{Case $r=2$}
Consider $\gs\in \cS(U\backslash D)$,
\[\gs=\sum_{i,j}(\sum_k g_{ijk}(z)z_1^{\al_{ijk}}z_2^{\be_{ijk}})(\log z_1)^i(\log z_2)^j\]
where $g_{ijk}(z)$ are meromorphic functions with poles along $\{z_1=0\}\cup\{z_2=0\}$. 
Then $\gs$ is of reduced form if
\begin{equation}\label{reduced two}
\re\al_{ijk}\in [0,1),\,\re{\be_{ijk}}\in [0,1);\;\text{when }k\neq k',\text{ either }\al_{ijk}\neq \al_{ijk'}\text{ or }\be_{ijk}\neq \be_{ijk'}.
\end{equation}
For each $i$, let $\{\al_{ijk}\}_{j,k}=\{\al_{i1},\ldots,\al_{i{s_i}}\}$ where $\al_{i1},\ldots,\al_{is_i}$ are pairwise distinct.
We can rewrite $\gs$ as
\[\gs=\sum_{i}(\log z_1)^{i}\bigg(\sum_{l_i=1}^{s_i} z_1^{\al_{il_i}}(\sum_{\{j,k\mid\al_{ijk}=\al_{il_i}\}} g_{ijk}(z)z_2^{\be_{ijk}}(\log z_2)^j)\bigg).\]

Suppose for some $b\in U\backslash D$, $\gs(b)=0$ for all $\gs\in\cS_b$. Then $z_1(b)z_2(b)\neq 0$. First we fix $z_i=z_i(b)$ for $2\leq i\leq n$ and consider the analytic continuation around $z_1=0$.
Then 
{\small\[\gs=\sum_{i}(\log z_1)^{i}\bigg(\sum_{l_i=1}^{s_i} z_1^{\al_{il_i}}(\sum_{\{j,k\mid\al_{ijk}=\al_{il_i}\}} g_{ijk}(z_1,z_2(b),\ldots, z_n(b))z_2(b)^{\be_{ijk}}(\log z_2(b))^j)\bigg).\]}
Let $\gs_{i,l_i}(z):=\sum_{\{j,k\mid\al_{ijk}=\al_{il_i}\}} g_{ijk}(z)z_2^{\be_{ijk}}(\log z_2)^j$,
then 
\begin{equation}\label{two variable}
\gs=\sum_{i}(\log z_1)^{i}\big(\sum_{l_i=1}^{s_i} z_1^{\al_{il_i}}\gs_{i,l_i}(z_1,z_2(b),\ldots, z_n(b))\big)
\end{equation}
and $\gs_{i,l_i}(z_1,z_2(b),\ldots, z_n(b))$ is a meromorphic function in $z_1$ with poles along $\{z_1=0\}$.
Then \eqref{two variable} satisfies \eqref{assump} and by case $r=1$ of Proposition \ref{general} we have
\[\gs_{i,l_i}(z(b))=\sum_{\{j,k\mid\al_{ijk}=\al_{il_i}\}} g_{ijk}(z(b))z_2(b)^{\be_{ijk}}(\log z_2(b))^j=0.\]
for all $i,l_i$.

Fix $i,l_i$. Note that if $k\neq k'$ and $\al_{ijk}=\al_{ijk'}=\al_{il_i}$, \eqref{reduced two} implies $\be_{ijk}\neq \be_{ijk'}$. Now in $\gs_{i,l_i}$ we fix $z_i=z_i(b)$ for $i\neq 2,1\leq i\leq n$  and do analytic continuation around $z_2=0$, then by case $r=1$ of Proposition \ref{general} again $\gs_{i,l_i}(z(b))=0$ implies $g_{ijk}(z(b))=0$ for all $j,k$ such that $\al_{ijk}=\al_{il_i}$. 

Hence if $\gs(b)=0$ for all $\gs\in \cS_b$, then $g_{ijk}^{(\gs)}(z(b))=0$ for all $g_{ijk}^{(\gs)}.$
Therefore Proposition \ref{general} holds for $r=2$.

\subsection{Analyticity of the zero locus}

Let $\cN:=\{b\in B\mid \gs(b)=0,\,\forall\gs\in\cS_b\}.$ Let  $\overline{\cN}$ denote its analytic closure in $\overline B$.
\begin{prop}\label{closure}
If an ALCS $\cS$ on $B$ has regular singularity along $D$, then $\overline{\cN}$ is analytic.
\end{prop}
\begin{proof}
$\gs$ is locally holomorphic away from $D$, thus $\cN$ is an analytic subvariety of $B$. In particular, $\cN$ is a closed subset of $B$.

Let $b_0\in D\cap\overline{\cN}$. Then by Proposition \ref{general} there exists a polydisk $U$ centered at $b_0$ such that  
$$\cN\cap (U\backslash D)=\{b\in B\mid g^{(\gs)}_{I k}(z(b))=0,\,\forall \gs\in\cS_b,\,\forall I,k\}\cap(U\backslash D)$$
where $g_{Ik}$ are meromorphic functions with poles along $\bigcup^r_{i=1}\{z_i=0\}$. Let $\chi_{Ik}^i\in \Z^r$ be the order of poles of $g_{Ik}(z)$ corresponding to $z_i$ respectively. Then $[z]_r^{\chi_{Ik}}g_{Ik}(z)$ is holomorphic on the neighborhood $U$. Then 
$$\overline{\cN}\cap U=\{b\in \overline B\mid [z(b)]_r^{\chi^{(\gs)}_{Ik}}g^{(\gs)}_{I k}(z(b))=0,\,\forall \gs\in\cS_b,\,\forall I,k\}\cap U,$$
i.e. $\overline{\cN}$ is analytic.
\end{proof}

\section{Algebraicity of $\cN(\delta)$}
As before, let $\tau$ be a regular holonomic tautological system on $V^\vee$, $B$ be a Zariski dense open subset of $V^\vee$, and $D=V^\vee\backslash B$.

By Hironaka's Theorem \cite{Hi} there exists a proper analytic morphism (blow-up) $f$:
\[\xymatrix{
&\tilde{V}^\vee\ar[r]^f &V^\vee\\
&\tilde B=\tilde{V}^\vee\backslash \tilde{D}\ar[u]^{\cup}\ar[r]^{\simeq}&B=V^\vee\backslash D\ar[u]_{\cup}
}\]
such that $\tilde{D}:={f}^{-1}(D)$ is a normal crossing divisor in $\tilde{V}^\vee$. We can then consider the $D$-module $\tilde\tau=f^*\tau$ on $\tilde{V}^\vee$ and its solution sheaf. Since $\tau$ is regular holonomic, $\tilde \tau$ is also regular holonomic.
Note that $f|_{\tilde{B}}$ induces an isomorphism from $\delta \sol(\tilde\tau)$ on $\tilde B$ to $\delta \sol(\tau)$ on $B$. Let $\tilde\cN(\delta):=\{b\in \tilde B\mid \tilde\gs(b)=0,\, \forall \tilde\gs\in \delta \sol(\tilde\tau)_b\}$.

{\claim{}{
The closure in analytic topology $\overline{\tilde\cN(\delta)}$ is analytic in $\tilde{V}^\vee$.
}}
\proof
Since $\tilde D$ is a normal crossing divisor and $\tilde\tau$ is regular holonomic, 
$\sol(\tilde\tau)$ has regular singularity along $\tilde D$. Then it is clear that $\delta \sol(\tilde\tau)$ also has regular singularity along $\tilde D$. Then by Proposition \ref{closure}, $\overline{\tilde\cN(\delta)}$ is analytic in $\tilde{V}^\vee$.
\qed

\begin{prop}\label{analy}
The closure in analytic topology $\overline{\cN(\delta)}$ is analytic in $V^\vee$.
\end{prop}
\begin{proof} First we claim two properties. \\
{\it{$f|_{\overline{\tilde\cN(\delta)}}$} is proper:}
Given a compact subset $C\subset V^\vee$, 
$$(f|_{\overline{\tilde\cN(\delta)}})^{-1}(C)={\overline{\tilde\cN(\delta)}}\cap f^{-1}(C).$$
Since $f$ is proper, $f^{-1}(C)$ is compact. Since ${\overline{\tilde\cN(\delta)}}$ is closed, ${\overline{\tilde\cN(\delta)}}\cap f^{-1}(C)$ is compact in ${\overline{\tilde\cN(\delta)}}$.\\
{\it{$f|_{\overline{\tilde\cN(\delta)}}$} is holomorphic:}
The restriction of a holomorphic map to an analytic space is holomorphic. 

Then by Proper Mapping Theorem (cf. \cite[p.162]{GR}) $f({\overline{\tilde\cN(\delta)}})$ is analytic.

Since $f$ is continuous, $f({\overline{\tilde\cN(\delta)}})\subset {\overline{f(\tilde\cN(\delta))}}$. On the other hand, given any sequence $\tilde x_k\in \tilde\cN(\delta)$ such that $\lim_{k\ra \infty}f(\tilde x_k)=y\in D$. We can take a compact neighborhood $C\subset V^\vee$ of $y$. Then for $k>>0$, $f(\tilde x_k)\in C$, i.e. $\tilde x_k\in f^{-1}(C)\subset \tilde V^\vee$. Since $f$ is proper, $f^{-1}(C)$ is compact. Thus there exists a convergent subsequence $x_{k'}$ such that $\lim_{k'\ra\infty} x_{k'}$ exists. Therefore by continuity of $f$ we have 
$$f(\lim_{k'\ra\infty} x_{k'})=\lim_{k'\ra\infty} f(x_{k'})=y$$
which means  $y\in f({\overline{\tilde\cN(\delta)}})$.
Thus
$$f({\overline{\tilde\cN(\delta)}})= {\overline{f(\tilde\cN(\delta))}}=\overline{\cN(\delta)}$$
and therefore $\overline{\cN(\delta)}$ is analytic.
\end{proof}

\begin{prop}
If $\delta\in D_{V^\vee}$ is homogeneous under scaling by $\C^\times$, $\overline{\cN(\delta)}$ is algebraic.
\end{prop}
\begin{proof} By Proposition \ref{analy}, $\overline{\cN(\delta)}\subset V^\vee=\C^n$ is closed analytic.
Suppose $\delta$ is homogeneous of degree $d$ under scaling by $\C^\times$. Given $\lambda\in\C^\times,$ for $\gs\in \sol(\tau)_b,$
\[(\delta \gs)(\lambda b)=\lambda^{d-\beta(e)}(\delta \gs)(b).\]
Thus $\lambda b\in\cN(\delta)$ if $b\in\cN(\delta)$, i.e.  the $\Cx$-action by scaling on $V^\vee$ leaves ${\cN(\delta)}$ invariant. Hence
$\Cx$ also leaves $\overline{\cN(\delta)}$ invariant.

Let $p:\C^n\backslash\{0\}\ra \P^{n-1}$ be the projection. Then $p(\overline{\cN(\delta)}\backslash\{0\})$ is a closed analytic subspace of $\P^{n-1}$, by Chow's theorem it is an algebraic subvariety. Thus its cone $\overline{\cN(\delta)}$ is an algebraic variety.
\end{proof}

Since $\cN(\delta)$ is a closed subset of $B$, $\cN(\delta)=\overline{\cN(\delta)}\cap B$.
\begin{thm}
If $\delta\in D_{V^\vee}$ is homogeneous  under scaling by $\C^\times$, $\cN(\delta)$ is algebraic.
\end{thm}

\section{Non-emptiness of $\cN(\delta)$: $\P^1$ case}

We now consider the problem of non-emptiness of $\cN(\delta)$, starting with the simplest nontrivial case when $X=\P^1,\,G=SL_2$. In this case $R\equiv\C[x_1^2,x_2^2,x_1x_2]$, $f=a_0x_1x_2+a_1x_1^2+a_2x_2^2$. Recall that 
\begin{align*}
Z(h)&=-2a_1\partial_1+2a_2\partial_2\\
Z(x)&=-2a_2\partial_0-a_0\partial_1\\
Z(y)&=-2a_1\partial_0-a_0\partial_2.
\end{align*}
for
$$h=\begin{pmatrix}
1&0\\
0&-1
\end{pmatrix},
x=\begin{pmatrix}
0&1\\
0&0
\end{pmatrix},
y=\begin{pmatrix}
0&0\\
1&0
\end{pmatrix}.$$

\begin{prop}
If $X=\P^1,\,G=SL_2$, given a positive integer $d$, then $\cN(\delta)\neq \emptyset$ for every $\delta\in\C[\partial]_d.$
\end{prop}
\begin{proof}
Step 1: $\gs\gl_2$ acts on $\C[\partial]_d$ by commutator $[Z(\xi),\delta]$ for $\xi\in\gs\gl_2$, $\delta\in\C[\partial]_d$. Since $\gs\gl_2$ is a semisimple Lie algebra and $\C[\partial]_d$ is a finite dimensional $\gs\gl_2$-module, $\C[\partial]_d$ is a semisimple $\gs\gl_2$-module. 

Step 2: Let $\Delta:=a_0^2-4a_1a_2.$ When $X=\P^1$, up to scalar the solution of $\tau$ is $\Delta^{-\frac{1}{2}}$. Define 
$$Ann_d:=Ann_{\C[\partial]_d}(\Delta^{-\frac{1}{2}}):= \{\alpha\in \C[\partial]_d\mid \alpha(\Delta^{-\frac{1}{2}})=0 \}.$$
Given $\alpha\in Ann_d,$ then $\alpha(\Delta^{-\frac{1}{2}})=0,$ thus $[Z(\xi),\alpha](\Delta^{-\frac{1}{2}})=0$ and $[Z(\xi),\alpha]\in Ann_d.$ Therefore $Ann_d$ is an $\gs\gl_2$-submodule of $\C[\partial]_d$.

Step 3: By Step 1, $\C[\partial]_d$ is a semisimple $\gs\gl_2$-module, then there exists an $\gs\gl_2$-submodule $S_d$ such that $\C[\partial]_d=Ann_d\oplus S_d$ as $\gs\gl_2$-modules.

Step 4: It is well known that $\gs\gl_2$-invariant ring is 
$$\{\alpha\in \C[\partial]_d\mid [Z(\xi),\alpha]=0 \;\forall \xi\in\gs\gl_2\}=\C[\partial]_d^{\gs\gl_2}=\C[\partial^2_0-\partial_1\partial_2]_d$$
 where $\C[\partial^2_0-\partial_1\partial_2]$ denotes the polynomial ring generated by a single element $\partial^2_0-\partial_1\partial_2$. It is clear that $\C[\partial]_d^{\gs\gl_2}\subset Ann_d$ for $d>0.$

Step 5: Let $\delta\in \C[\partial]_d$, we claim that $\cN(\delta)= \emptyset$ if and only if $\delta(\Delta^{-\frac{1}{2}})\in \C^\times\Delta^{-\frac{d+1}{2}}$. If $\delta(\Delta^{-\frac{1}{2}})\in \C^\times\Delta^{-\frac{d+1}{2}}$, then $\delta(\Delta^{-\frac{1}{2}})$ is nowhere vanishing. Thus $\cN(\delta)= \emptyset.$ For the other direction, we first observe that $$\delta(\Delta^{-\frac{1}{2}})=\Delta^{-\frac{1}{2}-d}P_d(a_0,a_1,a_2)$$ where $P_d$ is a homogeneous polynomial of degree $d$. Suppose $P_d$ factors into $P_d=\Delta^k q_{d-2k}$ where $\gcd(\Delta,q_{d-2k})=1$. Then  $\delta(\Delta^{-\frac{1}{2}})=\Delta^{-\frac{1}{2}-d+k} q_{d-2k}$. $\cN(\delta)=\emptyset$ implies that $\{\Delta^{-\frac{1}{2}-d+k} q_{d-2k}=0\}\cap\{\Delta\neq 0\}=\emptyset$ and thus $\{q_{d-2k}=0\}\subset\{\Delta= 0\}$. But $q_{d-2k}$ and $\Delta$ are coprime, it implies that  $q_{d-2k}\subset \Cx.$ Thus $d=2k$ and $\delta(\Delta^{-\frac{1}{2}})\in \C^\times\Delta^{-\frac{d+1}{2}}$.

Step 6:  Suppose $\cN(\delta)= \emptyset$, then by Step 5 we have $\delta(\Delta^{-\frac{1}{2}})\in \C^\times\Delta^{-\frac{d+1}{2}}$. Since $Z(\gs\gl_2)(\Delta^{-\frac{d+1}{2}})=0,$ it implies that $[Z(\gs\gl_2),\delta]\subset Ann_d.$ Step 3 tells us that $\delta=\delta'+\delta''$ where $\delta'\in Ann_d$, $\delta''\in S_d$ and
\[[Z(\xi),\delta]=[Z(\xi),\delta']+[Z(\xi),\delta''].\]
Since $[Z(\xi),\delta] \subset Ann_d$ and $[Z(\xi),\delta'] \subset Ann_d$, the direct sum forces  $[Z(\xi),\delta'']=0$ for all $\xi\in\gs\gl_2.$ This implies $\delta''\in \C[\partial]_d^{\gs\gl_2}\subset Ann_d$ when $d>0$. Then the direct sum further forces that $\delta''=0$. Thus $\delta=\delta'\in Ann_d.$ Therefore $\delta(\Delta^{-\frac{1}{2}})=0$, contradicts $\delta(\Delta^{-\frac{1}{2}})\in \C^\times\Delta^{-\frac{d+1}{2}}.$ 

Therefore given a positive integer $d$, $\cN(\delta)\neq \emptyset$ for every $\delta\in\C[\partial]_d.$
\end{proof}

\section{A degree bound}

In this section we consider $X=\P^{m},\, G=SL_{m+1}$. In this case, we will view $R=\C[a^\vee]$ as the subring of $\C[x_0,...,x_m]$ generated by the degree $m+1$ monomials in the $x_i$. This degree however will {\it not} be used below. The {\it degree} $\deg$ below shall refer to the degree in the variables $a_i^\vee$ which can be identified with a monomial basis of $V^\vee$.
 We now prove an important degree bound and use a rank approach to give another proof of $\cN(\delta)$ being algebraic.
\begin{lem}[Degree bound lemma]\label{dbl}
Take $X=\P^{m},\, \hat{\gg}=\gs\gl_{m+1}\oplus\C.$ Let $Z_i:=Z^\vee(x_i)$ where $x_i$ is a basis of $\hat{\gg}$. Suppose $f(b)$ is nonsingular. For $h\in R$,  $he^{f(b)}\equiv0$ in $H_0(\hat{\gg},Re^{f(b)})$ iff $$he^{f(b)}=\sum Z_i(r_i e^{f(b)})$$ for some $r_i\in R,$ and $\deg r_i\leq\deg h-1,\,\forall i.$
\end{lem}
\begin{proof}
The `if' direction is obvious. For the `only if' direction, consider the homogeneous ideal $I:=\bra x_u\partial_v {f(b)}|0\leq u,v\leq m\ket$ of $R$. Let $B_k$ denote a $\C$-basis for the degree $k$ part of $R/I$. First, since ${f(b)}$ is homogeneous of degree $1$, the degree $0$ part of $R/I$ is nonzero, and is spanned by $1$. For any $h\in R$, consider expanding the highest degree component of $h$, which we denote by $h_0$, in degree = $\deg h$ part of $R/I$ in terms of the chosen basis: i.e. by definition, there exist elements $s_i\in R$, such that $h_0-\sum s_iZ_i(f(b))$ can be written as a linear combination of the chosen basis elements in degree = $\deg h$. Obviously, we can require that $\deg s_i\leq \deg h-1$ for each $i$ by dropping all higher degree components of each of these $r_i$, if there are any. Working degree by degree, it is clear that we can choose $r_i\in R$ with $\deg r_i\leq \deg h-1, \forall i$, such that $he^{f(b)}=\sum Z_i(r_i e^{f(b)})+\sum c_kB_k$, where $\sum c_kB_k$ denote a linear combination of elements of the $B_k$ with all $k\leq \deg h$. Therefore, $H_0(\hat{\gg},Re^{f(b)})$ is spanned by $B_k$.

On the other hand, observed that $R/I=(\C[x_0,...,x_m]/J)^{\mu_{m+1}}$, where $J:=\bra\partial_i{f(b)}|0\leq i\leq m\ket$ is the Jacobian ideal of the nonsingular hypersurface ${f(b)}$, and $\mu_{m+1}$ is the group of $(m+1)$-th root of unity. By \cite{AS}\cite{G}, $\dim_{\C} (\C[x_0,...,x_m]/J)^{\mu_{m+1}}=h^{m}(X-V({f(b)}))$. Combining the algebraic and geometric rank formula for $\tau$, we have in this case, $h^{m}(X-V({f(b)}))=\dim H_0(\hat{\gg},Re^{f(b)})$. Therefore, the collection of $B_k$ consists of linearly independent elements, and $he^{f(b)}=0$ in $H_0(\hat{\gg},Re^{f(b)})$ iff all coefficients $c_k=0$.
\end{proof}

We now apply the lemma to derive explicit polynomial equations for the variety $\cN(\delta)$. Fix $\delta\in \C[\partial]$ be of degree $d$. 
Then $h e^f=\delta e^f$ where $h\in R$ such that $\tilde h=\delta$. By the lemma, each $b\in \cN(\delta)$ lies in the locus of an equation of the form
{\small\begin{equation}\label{expand}
 he^{f}=\sum Z_i(r_i e^{f})=\sum_i (Z_i(r_i) e^{f}+Z_i(f)r_i e^{f})
 \end{equation}}
for some $r_i=\sum_{j\in J}\lambda_j^i e_j,\,\lambda_j^i \in\C$, where $\{e_j\}_{j\in \cJ_{d-1}}$ being a basis of the subspace of $R$ of degree $\leq d-1$.  Thus $\deg Z_i(r_i)\leq d-1$ and $ Z_i(f)r_i$ is linear in the variables $a_i$, but is of degree $\leq d$ in the variables $a_i^\vee$.

In the basis $\{e_j\}_{j\in \cJ_d}$, $h$ can be viewed as a vector $\Theta\in\C^{\cJ_d}$. Let $\Lambda$ be the column vector with entries $\lambda^i_j,\,\forall i,j.$ Then comparing coefficients of the expansion of \eqref{expand} gives us a matrix $M_d(a)$ (depending only on $d$ but not on $\delta$ itself) whose entries lie in $\C+\sum_i\C a_i$ such that the following inhomogeneous linear system holds:
\begin{equation}\label{linear eqns}
M_d(b)\Lambda=\Theta.
\end{equation}
But in turn this is equivalent to the rank condition
\begin{equation}\label{rank cond}
\boxed{\rk M_d(b)=\rk[M_d(b)|\Theta].}
\end{equation}
To summarize, let's fix $\delta\in\C[\partial]$ (hence fix $h$ and $\Theta$) of degree $d$. For a given $b\in B$, Lemma \ref{dbl} says that  $b\in\cN(\delta)$ iff there exists $r_i\in R$ with $\deg r_i\leq d-1$ such that 
\begin{equation}\label{matrix}
he^{f(b)}=\sum Z_i(r_i e^{f(b)}).
\end{equation}
This is equivalent to saying that \eqref{rank cond} holds.  Thus we can conclude:

\begin{thm}
For a given $\delta\in\C[\partial]$ of degree $d$
\[\cN(\delta)=\{b\in B\mid \rk M_d(b)=\rk[M_d(b)|\Theta]\}.\]
Therefore $\cN(\delta)$ is an algebraic variety defined by the rank condition \eqref{rank cond}. In particular, $\cN(\delta)$ has a natural stratification given by $\rk M_d(b)$.
\end{thm}

\section{Periods of elliptic curves}

\subsection{Some preparation}
In this section we consider the case $X=\P^2$. $G=SL_3$. Then $\pi: \cY\ra B$ is the family of smooth elliptic curves in $X$.
We write the basis of $V$ as $\{a_I\mid I=(ijk), \,i+j+k=3,\, i,j,k\geq 0\}$, which is dual to the monomial basis $x_1^ix_2^jx_3^k$ of sections in $V^\vee$.
Let $S,T$ be Aronhold invariants of a ternary cubic, then 
$\C[V^\vee]^{SL_3}=\C[S,T].$ Let $\Delta=64S^3-T^2$ be the discriminant.

\begin{lem}
There is a natural action of $G$ on $B$. $B/G=\Spec \C[S,T,\Delta^{-1}]$. In particular, $S,T$ give a global coordinate system on the two dimensional nonsingular variety $B/G$.
\end{lem}
\begin{proof}
We have $B=\Spec \C[a_I,\Delta^{-1}]$, the stable locus of the $G$-action on $B$ (see \cite[Theorem 1.6]{N}). Thus every $G$-orbit in $B$ is closed, and we have 
$$B/G=\Spec \C [a_I,\Delta^{-1}]^G=\Spec \C[S,T,\Delta^{-1}],$$
 where it is well known that $S,T$ are algebraically independent.
\end{proof}
\begin{lem}\label{vanish lemma}
Let $\delta$ be a first order differential operator with constant coefficient (i.e. $\delta=\sum_{I} \lambda_I\frac{\partial}{\partial a_I}$, for constants $\lambda_I$). Let $h:=\delta f$ where $f:=\sum_{I}a_Ia_I^\vee$ is the universal section. Then given any point $b\in B$, the following are equivalent:
\begin{enumerate}
\item $b\in \cN(\delta)$, 
\item  $h=Z_xf(b)$ for some $x\in \gs\gl_3$,
\item $(\delta S)(b)=0$ and $(\delta T)(b)=0$.
\end{enumerate}
\end{lem}

\begin{proof}
Since $\delta$ is of degree 1, by Lemma \ref{dbl}, $b\in \cN(\delta)$ iff 
\begin{equation}\label{cond}
he^{f(b)}=\sum_{i=1}^8 Z_i(c_ie^{f(b)})+(E+1)(c_Ee^{f(b)})
\end{equation}
for some complex numbers $c_i$ and $c_E$, where $Z_i$ is a basis of $\gs\gl_3$, $E$ is the Euler operator. $h=\delta f$ implies that $h$ is a homogeneous polynomial in $a_I^\vee$ of degree 1 with constant coefficients, i.e. an element in the section space $V^\vee$. So \eqref{cond} holds iff
\begin{equation}\label{c1}
he^{f(b)}=\sum_{i=1}^8 Z_i(c_ie^{f(b)})
\end{equation}
for some complex numbers $c_i$. This is equivalent to
\begin{equation}\label{c2}
h=Z_xf(b)
\end{equation}
for some $x\in \gs\gl_3$.

Next, we identify $V^\vee$ with its tangent space at $b$, where $a_I^\vee$ is identified with $\frac{\partial}{\partial a_I}$.

Note that the identification is compatible with the action of $\gs\gl_3$. Under this identification, it is clear that $h$ is identified with $\delta$. ($\delta=\sum_{I}\lambda_I\frac{\partial}{\partial a_I}$ is identified with $\sum_{I} \lambda_Ia_I^\vee=h$.) We consider the projection map $p: B\mapsto B/G$. We denote the tangent map at $b$ by $dp_b$. At $b$, $dp_b(\delta)=0$ iff $\delta=h$ lies in the tangent space of the $G$-orbit $G\cdot b$, i.e. iff \eqref{c2} holds.

$S$ and $T$ are global coordinates of $B/G\subset\Spec (V^\vee)^G=\Spec\C[S,T]$. Thus, we have
{\small\begin{equation}\label{tangent}
dp_b(\delta)=\left.\frac{\partial}{\partial S}(\delta S)\right|_b+\left.\frac{\partial}{\partial T}(\delta T)\right|_b
\end{equation}}
So $dp_b(\delta)=0$ iff $(\delta S)(b)=0$ and $(\delta T)(b)=0$.
\end{proof}

This shows that
$$\cN(\delta)=\cM\cap B,\bl \cM:=\{b\in B\mid \delta S(b)=\delta T(b)=0\}\subset V^\vee.$$
In particular, this implies that $\cN(\delta)=\emptyset$ iff $\cM\subset\{\Delta=64S^3-T^2=0\}$. By Nullstellensatz, this is equivalent to 
$$\Delta\in\sqrt{\bra\delta S,\delta T\ket}$$
the radical of the ideal $\bra\delta S,\delta T\ket$. In other words
\begin{equation}\label{Null}
\Delta^m\in\bra \delta S,\delta T\ket
\end{equation}
for some integer $m>0$.

\begin{rem}
If we do not require $\delta$ to be constant coefficients, then $\cN(\delta)$ can be empty. E.g. Take $\delta$ to be Euler and $\beta\neq1$. Then $\cN(\delta)$ is the set $b$ where $\delta \gs(b)=\gs(b)=0$ for all periods $\gs$, hence empty because there is no point $b\in B$ where all periods vanish.
\end{rem}

\subsection{Main theorem}

\begin{thm}\label{main theorem}
Let $\delta=\sum_I \lambda_I\partial_I$, where $(\lambda_I)\in\Z[i]^{10}$ 
and 
{\small\eq{gcd assumption}{\gcd(\lambda_I)_I=1, ~(1+i)|\lambda_{111}\text{ in $\Z[i]$ and $\{\lambda_{300},\lambda_{030},\lambda_{003}\}\neq \{1,0,0\}\mod (1+i)$.}}}
Then $\cN(\delta)\neq\emptyset$.
\end{thm}

We prove this by a series of lemmas.

Recall that 
$$\C[V^\vee]^{SL_3}=\C[S,T]$$
where $S,T$ are the Aronhold invariants, which are respectively polynomials of degree 4 and 6 with integer coefficients in 10 variables. In order to use their explicit expressions given in \cite{S} which we include in Appendix \ref{appendix B}, we must multiply each variable $a_I=a_{ijk}$ appearing in our universal cubic section $f=\sum_I a_I x^I$ by the factor  $i! j! k!\over 3!$. {\it All use of the $a_I$ in this proof will be the $a_I$ defined in \cite{S}.}

Recall that the discriminant polynomial for the cubic plane curves is $\Delta=64S^3-T^2\in\Z[a]_{12}$. Evaluating $\Delta$ at the point $a_{111}=1$ and $a_I=0$ for $I\neq(111)$ yields $\Delta=0$, since this point defines the singular curve
$x_1x_2x_3=0$. Thus the monomial $a_{111}^{12}$ does not appear in the polynomial $\Delta$.

{\it Suppose $\cN(\delta)=\emptyset$.} Then by Nullstellensatz there exists a positive integer $m$ and
$$(h_S,h_T)\in W:=\C[a]_{12m-3}\oplus \C[a]_{12m-5}$$ 
such that \eqref{Null} becomes
\begin{equation}\label{*}
\Delta^m=(64S^3-T^2)^m=h_S\delta S+h_T\delta T.
\end{equation}
Fix an ordering of the monomial basis in the $a_I$ for each $\Q(i)[a]_k$, and so that we can now represent the polynomial $\Delta^m$ by the column $\Z$-vector $\theta$ given by the polynomial's coefficients, and $(h_S,h_T)$ by a column $\C$-vector $h$. Then \eqref{*} is equivalent to a matrix equation of the form
\begin{equation}\label{linear}
Mh=\theta
\end{equation}
where $M$ is a matrix over $\Z[i]$ defined by the expressions of $\delta S$ and $\delta T$.
This equation has a solution $h$ iff 
\begin{equation}\label{rk cond}
\rk{}_{\C}(M)=\rk{}_{\C}[M|\theta].
\end{equation}
But since $M$ and $[M|\theta]$ are defined over $\Q(i)$ this equation is equivalent to (because rank of a matrix remains the same under field extensions)
\begin{equation}
\rk{}_{\Q(i)}(M)=\rk{}_{\Q(i)}[M|\theta].
\end{equation}
Therefore, we can assume that
$(h_S,h_T)\in W_{\Q(i)}:=\Q(i)[a]_{12m-3}\oplus \Q(i)[a]_{12m-5}$.

Write $h_S=\frac{1}{pd}r_S$, $h_T=\frac{1}{qd}r_T$, where the coefficients of each of the polynomials $r_S,r_T\in\Z[i][a]$ have $\gcd$ 1, and $p,q,d\in\Z[i]$, with $\gcd(p,q)=1$ in $\Z[i]$. Then we get

\lemma{rSrT}{There exist $r_S,r_T\in\Z[i][a]$ each having coefficients with $\gcd$ 1, and $p,q,d\in\Z[i]$, with $\gcd(p,q)=1$ in $\Z[i]$, and $m\in\Z_{>0}$ such that
\begin{equation}\label{new}
pqd(64S^3-T^2)^m=qr_S\delta S+pr_T\delta T.
\end{equation}}

{\it Notations.} $e_I$ denotes the standard unit vector in $\Z^{10}$ corresponding to $I$. For $\mu=(\mu_I)\in(\Z_{\geq 0})^{10}$, write $a^\mu=\prod_I a_I^{\mu_I}$ and $|\mu|=\sum_I \mu_I$. Note that $a^\mu$ is invariant under the diagonal maximal torus of $SL_3(\C)$ iff the {\it index sum} $\sum_I\mu_I I$ of $\mu$ is equal to $|\mu|(1,1,1)$.
\begin{lem}\label{nonproportional}
Let $\mu,\mu'\in(\Z_{\geq 0})^{10}$ such that $|\mu|=|\mu'|=\ell$ and $a^\mu,a^{\mu'}$ are torus invariant. Then for any $I\neq I'$, if $\partial_Ia^\mu$ and $\partial_{I'}a^{\mu'}$ are nonzero, then they are not proportional (over any field).
\end{lem}
\begin{proof}
This is because the index sums of the monomials in $\partial_Ia^\mu$ and $\partial_{I'}a^{\mu'}$ are different:
\[(\ell-i,\ell-j,\ell-k)\neq (\ell-i',\ell-j',\ell-k')\]
for $I\neq I'$.
\end{proof}
We will apply this to the cases $\ell=4,6$. For homogeneous polynomial $P\in\C[a]$, denote by $c_\mu(P)$ the coefficient of the monomial $a^\mu$ in $P$, so that 
$$P=\sum_\mu c_\mu(P) a^\mu.$$

 \begin{lem}\label{gcd of delta S}
 Consider the degree 4 invariant polynomial $S$.
For each index $I$, there exists $\mu$ such that $c_\mu(S)=\pm 1$ and $\mu_I=1$. For each such pair $(I,\mu)$
$$c_{\mu-e_I}(\delta S)=\pm\lambda_I.$$ 
Therefore the $\gcd(c_\nu(\delta S))_\nu=1$. More generally, without assuming $\gcd(\lambda_I)_I=1$, we also have
$$\gcd(c_\nu(\delta S))_\nu | \gcd(\lambda_I)_I.$$
\end{lem}
\begin{proof}
The explicit expression of $S$ shows the first statement holds. For the second statement, consider
$$\delta S=\sum_{I',\mu'} c_{\mu'}(S)\lambda_{I'}\partial_{I'} a^{\mu'}=\sum_{I',\mu'} c_{\mu'}(S)\lambda_{I'}\mu_{I'}' a^{\mu'-e_{I'}}.$$
For the given pair $(I,\mu)$, the summands on the right that are proportional to the monomial term $c_{\mu}(S)\lambda_I\partial_I a^\mu=\pm\lambda_I a^{\mu-e_I}$ must be those with $\mu'-e_{I'}=\mu-e_I$. But Lemma \ref{nonproportional} forces $I=I'$, hence $\mu=\mu'$. This proves the second statement. Finally, the last statement follows from the second and the assumption \eqref{gcd assumption}.
\end{proof}

Now consider the explicit expression $T\in\Z[a]$. There are exactly 6 monomial terms with odd coefficients, namely
\begin{equation}\label{T0}
\begin{split}
T_0:=&a_{300}^2a_{030}^2a_{003}^2-3a_{300}^2a_{021}^2a_{012}^2-3a_{030}^2a_{201}^2a_{102}^2\\ 
&-3a_{003}^2a_{210}^2a_{120}^2-27a_{201}^2a_{120}^2a_{012}^2-27a_{210}^2a_{102}^2a_{021}^2
\end{split}
\end{equation}
so that $T$ has the form $T=T_0+2T_1$ with $T_1\in\Z[a]$. 

(1) Note that $c_\nu(\delta T_0)\neq0$ implies that $\nu$ satisfies the index sum condition that
$$\sum_I\nu_I I=(6,6,6)-(i,j,k),~~~\text{ for some $i,j,k\geq0$ with $i+j+k=3$.}$$
If $\nu=\mu-e_I$ then this condition uniquely determines $\mu$ and $I$, with $\sum_I\mu_I I=(6,6,6)$.

(2) Each $a_I\neq a_{111}$ appears in some monomial $a^\mu$ in $T_0$ and with exponent 2. 

(3) Since $T=T_0+2T_1$, by Lemma \ref{nonproportional}, for $a_I\neq a_{111}$
\eq{tem1}{c_{\mu-e_I}(\delta T_0)=c_{\mu}(T_0)\lambda_I \mu_I}
where $\mu,I$ are uniquely determined by a given $\nu=\mu-e_I$.

(4) By the same lemma
$$c_{\mu-e_I}(\delta T)=
\begin{cases}
c_{\mu}(T_0)\lambda_I\mu_I~\text{ or } 2c_\mu(T_1)\lambda_I\mu_I & I\neq(1,1,1)\\
2c_{\mu}(T_1)\lambda_I \mu_I & I=(1,1,1)
\end{cases}$$
where $\mu,I$ are uniquely determined by a given $\nu=\mu-e_I$. Note that if $c_{\mu-e_I}(\delta T)=c_{\mu}(T_0)\lambda_I\mu_I\neq0$ then $\mu_I=2$.

\begin{lem}\label{delta T}
 The prime $1+i$ appears in prime factorization of
$\gcd(c_\nu(\delta T))_\nu$ in $\Z[i]$ with exponent exactly 2.
\end{lem}
\begin{proof}
By assumption \eqref{gcd assumption}, in $\Z[i]$
$$(1+i)\nmid\gcd(\lambda_I)_{I\neq(1,1,1)}.$$
Trivially in $\Z[i]$
$$2=(1+i)^2(-i),\hskip.2in(1+i)\nmid n,~~n\in2\Z+1.$$
Note that by (4), for $\nu=\mu-e_I$, $(1+i)^2|c_\nu(\delta T)$ for all $\nu$. So, it remains to show that $(1+i)^3\nmid c_\nu(\delta T)$ for some $\nu$. Pick $I\neq(1,1,1)$ such that $(1+i)\nmid\lambda_I$. By (2) and (4), we can find $\mu$ such that 
$$c_{\mu-e_I}(\delta T)=2c_{\mu}(T_0)\lambda_I\neq0.$$
Since $c_\mu(T_0)$ is odd, this number is not divisible by $(1+i)^3$.
\end{proof}

\lemma{p=1}{$p$ is a unit in $\Z[i]$. Therefore there exists $r_S,r_T\in\Z[i][a]$ each $r_S,r_T$ having coefficients with $\gcd$ 1, such that
$$\boxed{qd(64S^3-T^2)^m=qr_S\delta S+r_T\delta T.}$$}
\proof
Consider \eqref{new} in Lemma \ref{rSrT}. Since $\gcd(p,q)=1$, if $p_0$ is a prime factor of $p$ in $\Z[i]$, then the right side $\mod p_0$ is $r_S\delta S\mod p_0$, which is nonzero since the coefficients of $r_S$ have $\gcd$ 1 by assumption, and likewise for $\delta S$ by Lemma \ref{gcd of delta S}. But the left side of \eqref{new} is zero $\mod p_0$, a contradiction. This shows that $p$ is a unit in $\Z[i]$ which we can assume it is 1, by absorbing $p^{-1}$ into $r_S,r_T$.
\qed

\lemma{}{$(1+i)^3\nmid q$.}
 \proof For otherwise Lemma \ref{p=1} implies that $(1+i)^3|r_T\delta T$. Since $(1+i)^2|\delta T$ but $(1+i)^3\nmid\delta T$ by Lemma \ref{delta T}, it follows that $(1+i)|r_T$, contradicting that $r_S\in\Z[i][a]$ has coefficients with $\gcd$ 1. 
 \qed

\lemma{2 not divide q}{$(1+i)^2\nmid q$. }  
\proof
Suppose $2i=(1+i)^2$ divides $q$. Then we can write $q=2q_1$ with $q_1\in\Z[i]$ with $\gcd(q_1,1+i)=1$. 

Since $\half\delta T\in\Z[i][a]$, we can divide equation in Lemma \ref{p=1} by 2 and get
\begin{equation}\label{q1}
q_1 d \Delta^m=q_1r_S\delta S+r_T{\delta T\over 2} \in\Z[i][a].
\end{equation}

First we consider this equation in the residue field $\mod(1+i)$, which is $\Z/2\Z[a]=\F_2[a]$. Since $q_1\equiv1$, $64\equiv0$, $T\equiv H^2$, we get
\begin{equation}\label{red mod 2}
dH^{4m}\equiv r_S\delta S+r_T\frac{\delta T}{2}
\end{equation}
where
\ea{H&=&a_{300}a_{030}a_{003}+a_{300}a_{021}a_{012}+a_{030}a_{201}a_{102}\\
&&+a_{003}a_{210}a_{120}+a_{201}a_{120}a_{012}+a_{210}a_{102}a_{021}.}
This follows from direct observation that
$$\boxed{\Delta=T^2=(\partial_{111}S)^2=H^4\mod2.}$$

\claim\label{H irreducible}{$H$ is irreducible in $\F_2[a]$.}
\proof
For otherwise, we can factorize
$$H=(P_1 a_{300}+P_2)P_3$$
with the $P_i\in\F_2[a]$ independent of $a_{300}$ but $P_3$ is nonconstant, and $P_1P_3=a_{030}a_{003}+a_{021}a_{012}$. Check easily that this implies that $P_1$ is constant, so that we can set $P_1=1$. Therefore $P_2$ has degree 1 and
{\small$$P_2(a_{030}a_{003}+a_{021}a_{012})=a_{030}a_{201}a_{102}
+a_{003}a_{210}a_{120}+a_{201}a_{120}a_{012}+a_{210}a_{102}a_{021}$$}
which is clearly impossible.
\qed

Now $\mod2$ (hence also $\mod(1+i)$), the explicit expression of $S$ gives
$$S\equiv H a_{111}+a_{111}^4+P$$
for some $P\in\F_2[a]$ independent of $a_{111}$ and has $\deg P=4$. Setting 
$Q=\lambda_{111} H+\delta P$, we get

{\claim\label{delta S mod 1+i}{There exists $Q\in\F_2[a]$ independent of $a_{111}$ such that
$$\delta S\equiv(\delta H) a_{111}+Q\mod(1+i).$$}}

Next, recall from \eqref{T0} that $T=T_0+2T_1$. The explicit expression of $T_1$ has the form $T_1=T_2+T_3$ where $T_2\in\Z[a]$ depends on $a_{111}$  and $2|T_2$, and $T_3\in\Z[a]$ is independent of $a_{111}$. Then
$$\half\delta T=\half\delta T_0+\delta T_2+\delta T_3.$$
Note that $\half\delta T_0\in\Z[a]$ is also independent of $a_{111}$. Taking $\mod (1+i)$, $\delta T_2$ drops out since $2|T_2$. This shows 

{\claim{}{$\half\delta T\mod(1+i)\in\F_2[a]_5$ is independent of $a_{111}$, and it is nonzero by Lemma \ref{delta T}.}}

Suppose $(1+i)| d$. 
Let $u:=\gcd(\delta S,{\delta T\over 2})\in\F[i][a]$. By Lemma \ref{gcd of delta S} $\gcd(c_\nu(\delta S))_\nu=1$, thus $u\neq 0\mod (1+i)$. Then we have $P_S,P_T\in \Z[i][a]$  such that $\delta S=uP_S,~~{\delta T\over 2}=uP_T.$
Then in $\F_2[a]$ equation \eqref{q1} becomes
$$0\equiv q_1r_S\delta S+r_T{\delta T\over 2}\equiv u(q_1r_SP_S+r_TP_T).$$
Thus 
$$0\equiv q_1r_SP_S+r_TP_T\equiv r_SP_S+r_TP_T.$$

Since $(P_S,P_T)=1\in\Z[i][a]$, there exists $\al,\be\in\Z[i][a]$ such that $\al P_S+\be P_T=1$. Thus $\al P_S+\be P_T\equiv1\mod (1+i)$, i.e. $(P_S,P_T)= 1\in\F_2[a]$.
Therefore there exists some $h\in \F_2[a]$ such that $r_S\equiv h P_T$ and $r_T\equiv -h P_S\equiv -q_1h P_S$.

Now we take $h_0\in\Z[i][a]$ to be any lift of $h$ in $\Z[i][a]$. Then 
$$r_S= h_0P_T+r_S'\text{ and }r_T= -q_1h_0P_S+r_T'$$ for some $r_S',~r_T'\in \Z[i][a]$ and $(i+1)| r_S',~(i+1)| r_T'$.
Then equation \eqref{q1} becomes
{\small\begin{align*}
q_1 d \Delta^m&=q_1(h_0P_T+r_S')\delta S+(-q_1h_0P_S+r_T'){\delta T\over 2}  \\
&=q_1r_S'\delta S+r_T'{\delta T\over 2}\in\Z[i][a].
\end{align*}}
Now we can see that $(i+1)$ is a common factor on both sides, so we have
$$q_1 {d\over 1+i} \Delta^m=q_1{r_S'\over 1+i}\delta S+{r_T'\over 1+i}{\delta T\over 2} \in\Z[i][a].$$

Since $d\neq0$, there exists a largest positive integer $k$ such that
$$d_1={d\over(1+i)^k}\in\Z[i],\quad (1+i)\nmid d_1.$$
Then we can repeat the process and get 
\begin{equation}\label{reduced d}
q_1 d_1 \Delta^m=q_1r_S''\delta S+r_T''{\delta T\over 2} \in\Z[i][a].
\end{equation}
for some $r_S'',r_T''\in\Z[i][a]$. 
Note that $r_S'',r_T''$ not necessarily have the property that their coefficients have $\gcd$ 1. 

Claim \ref{delta S mod 1+i} shows that 
$$\delta S\equiv(\delta H) a_{111}+Q\mod(1+i)$$
with $Q=\lambda_{111} H+\delta P$. 
By \eqref{gcd assumption}, $\lambda_{111}\equiv 0\mod(1+i)$, then
\[\delta S\equiv(\delta H) a_{111}+\delta P\mod(1+i).\]
By looking at the explicit expression of $T$, we observe that $T\equiv H^2\mod 4.$ It implies that ${\delta T\over 2}\equiv  {H\delta H}\mod 2$ and hence
$${\delta T\over 2}\equiv  {H\delta H}\mod (1+i).$$
Then \eqref{reduced d} becomes
\[H^{4m}\equiv r_S''((\delta H) a_{111}+\delta P)+r_T''(H\delta H)\mod (1+i).\]
Now we evaluate both sides at $a_{111}=a_{120}=a_{102}=a_{210}=a_{012}=a_{021}=a_{201}=0.$ Let $Q|$ denotes the evaluation of $Q$ under this condition.
We observe that every term in $P$ contains at least two of these $a_I$'s, so $(\delta P)|=0$.
Then we get
\eq{evaluation}{(a_{300}a_{030}a_{003})^{4m-1}\equiv r_T''|(\delta H)|.}
We observe that 
$$(\delta H)|=\lambda_{300}a_{030}a_{003}+\lambda_{030}a_{300}a_{003}+\lambda_{003}a_{300}a_{030}.$$

Under condition \eqref{gcd assumption}, there are three types remaining:
\begin{enumerate}
\item  $\{\lambda_{300},\lambda_{030},\lambda_{003}\}\equiv \{0,0,0\}\mod (1+i)$.
Then $(\delta H)|\equiv 0\mod (1+i)$ and it contradicts \eqref{evaluation}.

\item $\{\lambda_{300},\lambda_{030},\lambda_{003}\}\equiv \{1,1,1\}\mod (1+i)$. Then
$$(\delta H)|=a_{030}a_{003}+a_{300}a_{003}+a_{300}a_{030}.$$
It is irreducible in $\F_2[a_{300},a_{030},a_{003}]$ and it does not divide $(a_{300}a_{030}a_{003})^{4m-1}$, it contradicts the fact that $\F_2[a_{300},a_{030},a_{003}]$ is an UFD.

\item $\{\lambda_{300},\lambda_{030},\lambda_{003}\}\equiv \{1,1,0\}\mod (1+i)$.
If $\lambda_{300}=1,\lambda_{030}=1,\lambda_{003}=0$, 
$$(\delta H)|=a_{030}a_{003}+a_{300}a_{003}=(a_{030}+a_{300})a_{003}.$$
But $a_{030}+a_{300}$ doesn't divide $(a_{300}a_{030}a_{003})^{4m-1}$, it contradicts the fact that  the fact that $\F_2[a_{300},a_{030},a_{003}]$ is an UFD.
\end{enumerate}

This shows that our initial supposition that $(1+i)^2|q$ is false, hence proving Lemma \ref{2 not divide q}.
\qed

\lemma{1+i divides q}{$(1+i)|q$.}
\proof 
Suppose not.

\claim\label{delta S=H}{$\delta S=H\mod(1+i).$}
\proof We have
\begin{equation}\label{red mod 2'}
qdH^{4m}=qr_S\delta S+r_T\delta T.
\end{equation}
Therefore in $\F_2[a]$, we have 
$$dH^{4m}=r_S\delta S$$ 
since $2|\delta T$. Since the coefficients of $r_S$ have $\gcd$ 1 by assumption, and same for $\delta S$ by Lemma \ref{gcd of delta S},  the right side is nonzero, hence $d$ is coprime to $(1+i)$, and hence $H=\delta S$ in $\F_2[a]$ (because the only unit in $\F_2$ is 1).
\qed

\claim\label{delta S=0}{Without assuming $\gcd(\delta):=\delta(\lambda_I)_I=1$, if $\delta S=0\mod(1+i)$ then $\delta=0\mod(1+i)$.}
\proof
We have $\delta S=0\mod(1+i)$ iff $(1+i)|\gcd(\delta S)$. Thus $(1+i)|\gcd(\lambda_I)$,  by Lemma \ref{gcd of delta S}, hence $\delta=0\mod(1+i)$.
\qed

\claim\label{delta-111}{$\delta=\partial_{111}\mod(1+i)$.}
\proof
The explicit expression of $S$ yields $\partial_{111}S=H\mod(1+i)$. So by Claim \ref{delta S=H}
$$(\delta-\partial_{111})S=H-H=0\mod(1+i).$$
Now letting  $\delta-\partial_{111}$ play the role of $\delta$ in Claim \ref{delta S=0}, implies the claim.
\qed

To finish the proof of Lemma \ref{1+i divides q}, observe that Claim \ref{delta-111} contradicts \eqref{gcd assumption}. This shows that the supposition that $(1+i)\nmid q$ is false.
\qed

\lemma{}{$q$ does not exists, hence Theorem \ref{main theorem} is proved.}
\proof
Lemmas \ref{1+i divides q} and \ref{2 not divide q} imply $(1+i)|q$ and $(1+i)^2\nmid q$. Lemma \ref{rSrT} gives
$${q\over1+i}d\Delta^{4m}={q\over 1+i}r_S\delta S+r_T{\delta T\over1+i}.$$
Since $(1+i)|{\delta T\over1+i}$, taking $\mod(1+i)$ yields
$$dH^{4m}=r_S\delta S\mod(1+i).$$
Again, since $\gcd(r_S)=1$ and $\delta S\neq0\mod(1+i)$, $d\neq0$ hence $H=\delta S\mod(1+i)$ as in Claim \ref{delta S=H}, hence
$$\delta=\partial_{111}\mod(1+i)$$
as in Claim \ref{delta-111}, which contradicts  \eqref{gcd assumption} again. 
\qed

\begin{prop}
The set of $\delta$ where $N(\delta)$ is nonempty, is dense in $V^\vee$, in analytic topology. Hence it is dense in Zariski topology.
\end{prop}
\begin{proof}
Let 
{\small\begin{align*}
S:=\{(\lambda_I)\in\Z[i]^{10}\mid \gcd(\lambda_I)_I=1, ~(1+i)|\lambda_{111}\text{ in }\Z[i]\\
\text{ and }\{\lambda_{300},\lambda_{030},\lambda_{003}\}\neq \{1,0,0\}\mod (1+i)\}.
\end{align*}}
Then we have shown in Theorem \ref{main theorem} that $\cN(\delta)\neq\emptyset$ for all $\lambda\in S$. We are going to show that given any point $\delta_{\bar\lambda}:=\sum_I\bar\lambda_I\partial_I\in V^\vee,$ we can find a sequence $\lambda^k\in\Q(i)^{10}$ such that $\lim_{k\ra\infty}\lambda^k=\bar\lambda$ and $\cN(\delta_{\lambda^k})\neq\emptyset$ for all $k$.

We consider a subset of $S_0\subset S$:
{\small\begin{align*}
S_0:=\{&(\lambda_I)\in\Z[i]^{10}\mid \gcd(\lambda_I)_I=1, \lambda_{300}\equiv\lambda_{030}\equiv\lambda_{003}\equiv 1\mod (1+i) \\
&\text{ and }\lambda_I\equiv 0\mod (1+i) \text{ for  }I\neq 300,030,003\}.
\end{align*}}
Since $\Q(i)^{10}$ is dense in $\C$, we can find a sequence $x^k\in\Q(i)^{10}$ such that $\lim_{k\ra\infty} x^k=\bar\lambda.$ For each $k$, we choose some $q^k\in\Z$ such that 
$q^kx^k\in\Z[i]^{10}$
and $\lim_{k\ra\infty}q^k=\infty.$

Then we look at each entry, say $I=111$. If $q^kx^k_{111}\equiv 0\mod (1+i)$, let $\lambda^k_{111}=x^k$; if $q^kx^k_{111}\equiv 1\mod (1+i)$, let $\lambda^k_{111}=\dfrac{q^kx^k_{111}+1}{q^k}$. We repeat this process to make each entry of $q^k\lambda^k$ satisfy the $\mod (1+i)$ condition in $S_0$. 
Then it is clear that $\lim_{k\ra\infty} |x^k-\lambda^k|=0$ and thus $$\lim_{k\ra\infty} \lambda^k=\lim_{k\ra\infty} x^k=\bar\lambda.$$
Note that $q^k\lambda^k$ is not necessarily in $S_0$ since it may not satisfy the $\gcd$ condition. Let $d^k=\gcd(q^k\lambda^k_I)_I\in\Z[i].$ Then by our construction $d^k\equiv 	1\mod(1+i)$. Then we consider $\dfrac{q^k\lambda^k}{d^k}\in\Z[i]^{10}$. It is clear that $\dfrac{q^k\lambda^k}{d^k}\in S_0$ and Theorem \ref{main theorem} implies that $\cN(\delta_{\frac{q^k\lambda^k}{d^k}})\neq\emptyset$. Since $\delta_{\frac{q^k\lambda^k}{d^k}}$ is homogeneous, $\cN(\delta_{\lambda^k})=\cN(\delta_{\frac{q^k\lambda^k}{d^k}})\neq\emptyset$, as desired.
\end{proof}

\begin{cor}
There exists a nonempty Zariski open subset $U_0\subset V^\vee$, such that for each $\delta\in U_0$, $N(\delta)\neq\emptyset$.
\end{cor}

\begin{proof}
Consider the projection morphism of schemes of finite type over $\C$:
$$f: \Spec \C[\lambda_I,a_I,\Delta^{-1}]/\bra\sum_I\lambda_I\partial_I S, \sum_I\lambda_I\partial_I T\ket\rightarrow \Spec \C[\lambda_I].$$
 $Im(f)$ contains a dense subset of $V^\vee$ in analytic topology and therefore also in Zariski topology, so $f$ is dominant, which implies that $Im(f)$ contains a non-empty Zariski open subset $U_0$, and consequently the corollary holds.
\end{proof}

\subsection{Another proof}
For the case $X=\P^2$, there is another simple proof for $N(\delta)\neq\emptyset$ where $\delta$ is a first order homogeneous constant coefficient differential operator. However, the proof cannot be generalized to higher dimension.

\begin{prop}
For $h=\delta$ (homogeneous, 1st order, constant coefficient) GIT-stable (i.e. in this case, smooth), and for each smooth section $f(b)$, we have $N(\delta)\cap G\cdot f(b)\neq\emptyset$ where $G\cdot f(b)$ denotes the $G$-orbit of $f(b)$. So in particular, $N(\delta)\neq\emptyset$.
\end{prop}
\begin{proof}
Since $h$ is GIT-stable, $h$ has finite stabilizer in $\P V^\vee$, under the action of $G=SL_3$. Therefore, the $G$-orbit of $h$ in $\P V^\vee$ is a closed subvariety of dimension 8. For the same reason, the $G$-orbit of $f(b)$ in $\P V^\vee$ has dimension 8, so $f(b)$ is not killed by any nonzero Lie algebra element in $\gs\gl_3$, (otherwise the exponential map would give rise to a one-parameter subgroup infinite stabilizer of $f(b)$, under the action of $G$) and therefore the $\C$-vector space $W_b:=\{Z_xf(b)|x\in \gs\gl_3\}$ has dimension 8. Therefore the projectivization $\P W_b$ is a closed subvariety of dimension 7 in $\P V^\vee$, which then must intersect with the $G$-orbit of $f(b)$ in $\P V^\vee$ by dimension reason. Therefore, there exists $g\in SL_3$, $\lambda\in\C, \lambda\neq 0$, and $x\in \gs\gl_3$ such that 
\begin{equation}
gh=\lambda Z_xf(b).
\end{equation}
Therefore, since $g^{-1}f(b)=f_{g^{-1}b}$, we have
\begin{equation}
h=g^{-1}Z_{\frac{1}{\lambda}x}gf_{g^{-1}b}.
\end{equation}
Since $g^{-1}Z_{\frac{1}{\lambda}x}g=Z_{x'}$ for some $x'\in \gs\gl_3$, Lemma \ref{vanish lemma} implies that $g^{-1}b\in\cN(\delta)$. Since $f_{g^{-1}b}\in G\cdot f(b)$, we have $\cN(\delta)\cap G\cdot f(b)\neq\emptyset$.
Hence the lemma follows.
\end{proof}

\section{An application to classical invariant theory}
Let $X=\P^{n-1}$ with $n\geq 3$, $V^\vee=\Gamma(X,K_X^{-1})$, and $G=SL_n$ as before. In this section, we prove the following
\begin{thm}\label{inv deg}
Let $\langle S_1,...,S_w\rangle$ be a system of homogeneous polynomials that generate $\C[V^\vee]^G$, then there exists an $S_k$ among these generators, such that $deg(S_k)\equiv 1 (\mod n)$.
\end{thm}

We first prove the following lemma for any $X=G/P$:
\begin{lem}\label{equivalences}
Let $\delta$ be a first order constant coefficient homogeneous differential operator, and $h=\delta f$ as before. Let $b\in B$. Then the following conditions are equivalent:
\begin{enumerate}
\item $b\in \cN(\delta)$.
\item $h=Z_xf(b)$ for some $x\in \gg$.
\item $(\delta P)(b)=0$ for any $P\in\C[V^\vee]^G$.
\end{enumerate}
\end{lem}
\begin{proof}
We already proved that (1) and (2) are equivalent in Lemma \ref{vanish lemma}. We now prove that (2) and (3) are equivalent. Again consider the projection morphism $p: B\to B/G$. By GIT theory, the function ring of $B/G$ is identified with $\C[V^\vee, \Delta^{-1}]^G$: i.e. elements in $\C[V^\vee]^G$ divided by powers of $\Delta$.

Assuming (3), take any regular function $\phi: B/G\to\C$, then $\phi\cdot p\in C[V^\vee, \Delta^{-1}]^G$, and therefore $(\delta (\phi\cdot p)) (b)=0$. (Note that $(\delta(\Delta))(b)=0$ as $\Delta\in\C[V^\vee]^G$.) So $(dp_b(\delta)\phi)(p(b))=(\delta (\phi\cdot p)) (b)=0$, where again $dp_b$ denotes the tangent map induced by $p$ at $b$. Therefore $dp_b(\delta)=0$, which is equivalent to (2) as we already know.

Assuming (2), for any $P\in\C[V^\vee]^G$, by abuse of notation we still denote $P$ restricting to $B$ by $P$. Then $P$ is a regular function on $B$ invariant under $G$, therefore $P=\phi_0\cdot p$ for a regular function $\phi_0$ on $B/G$. So (2) implies $dp_b(\delta)=0$, which in turn implies $(dp_b(\delta)\phi_0)(p(b))=0$. i.e. $(\delta P)(b)=0$.
\end{proof}

Now we prove Theorem \ref{inv deg}:
\begin{proof}
Let $X=\P^{n-1}$ and let $\delta=\partial_{a_{1...1}}$, so $h=x_1...x_n$. Let $b=x_1^n+...+x_n^n$ be the Fermat point, which lies in $B$. As $n\geq 3$, it is clear that $h$ does not satisfy condition (2) in lemma \ref{equivalences}. Therefore Lemma \ref{equivalences} implies that there exists a homogeneous element $S\in\C[V^\vee]^G$, such that $(\delta S)(b)\neq 0$. i.e. $(\partial_{a_{1...1}}S)(b)\neq 0$. This implies that $S$ contains a monomial term that is linear in $a_{1...1}$, which is a product of $a_{1...1}$ with powers of $a_{n0...0}, ..., a_{0...0n}$. Since any monomial in $S$ is invariant under the maximal torus action, for any monomial that appears in $S$, the sum of indexes at each position has to be equal.  Therefore, this monomial is a nonzero multiple of $a_{1...1}(a_{n0...0}...a_{0...0n})^k$ for some nonzero $k\in\N$ (as it is clear that there is no invariant polynomial in degree 1). 

Now, take $S'$ to be an element in $\C[V^\vee]^G$ such that it contains a monomial term that is a nonzero multiple of $a_{1...1}(a_{n0...0}...a_{0...0n})^k$ with minimal $k$. Then it is clear that $S'$ can not be written as a polynomial of invariant polynomials which do not contain monomial terms of this form. The theorem is therefore proved.
\end{proof}

\begin{rem}
It is possible to elaborate on this argument to extract further information about the invariant ring $\C[V^\vee]^G$, and to establish further relations between $\cN(\delta)$ and the invariant ring. Indeed, theorem \ref{inv deg} does not hold for $n=2$, precisely because in that case, the Fermat point does lie in $\cN(\delta)$ for the $\delta$ in the above proof.
\end{rem}

\begin{rem}

 There are indications that our study of $\cN(\delta)$ for 1st order derivatives is also related with the local Torelli theorem, as the vanishing loci of such derivatives of periods correspond to degenerations of the period map. It would also be interesting to investigate the invariant theoretic or geometric meaning of $\cN(\delta)$ for higher order $\delta$. We plan to study these questions in a future paper.
\end{rem}

\appendix

\section{Some examples for $\P^m$}\label{appendix A}
Making use of methods in \cite{BHLSY}, we can compute a basis of $\hat{\gg}Re^{b}$ explicitly at the large complex structure limit (LCSL) $b_\infty$ and the Fermat point $b_F$ for $\P^1$ and $\P^2$. By Theorem \ref{equiv}, this allows us to find 
explicit differential relations, i.e. linear relations for constant coefficient differential operators that kill periods at these points.

If $X=\P^m, \,G=SL_{m+1},$ then we can identify $R$ with the subring of $\C[x_0,\ldots,x_m]$ generated by degree $m+1$ monomials.
\begin{lem}{\cite[Lemma 2.12]{BHLSY} }\label{basis of g}
We have
\[\hat{\gg}\cdot(Re^{f(b)})=Re^{f(b)}\cap\sum_i\frac{\partial}{\partial x_i}(\C[x]e^{f(b)})\]
for all $b\in B$.
\end{lem}

\textbf{Computation for $\P^1$ at LCSL.} For $X=\P^1, \,G=SL_2,\,R\equiv\C[x_1^2,x_2^2,x_1x_2]$, $f=a_0x_1x_2+a_1x_1^2+a_2x_2^2$ and $b_\infty=x_1x_2$.

\begin{claim}
For integers $\alpha,\beta\geq 0$, $\alpha\neq\beta,\, \alpha+\beta>0$ and $2 |(\alpha+\beta)$, 
\[x_1^\alpha x_2^\beta e^{x_1x_2}\in \hat{\gg}\cdot(Re^{x_1x_2}).\]
\end{claim}
\begin{proof}
Without loss of generality we assume $\al>\be\geq 0$. 
For $m,n\geq 0,$ we observe
\begin{equation}\label{inductionP1LSCL}
\frac{\partial}{\partial x_2}(x_1^{m}x^{n+1}_2 e^{x_1x_2})=x_1^{m+1}x^{n+1}_2 e^{x_1x_2}+(n+1)x_1^{m}x^{n}_2 e^{x_1x_2}.
\end{equation}
Since $x_1^{\al-\be}e^{x_1x_2}=\frac{\partial}{\partial x_2}x^{\al-\be-1}_1e^{x_1x_2}$, then by \eqref{inductionP1LSCL}, $$x_1^{\al}x_2^\be e^{x_1x_2}\in \sum_i\frac{\partial}{\partial x_i}(\C[x]e^{x_1x_2}).$$
Thus by Lemma \ref{basis of g}, $x_1^\alpha x_2^\beta e^{x_1x_2}\in \hat{\gg}\cdot(Re^{x_1x_2})$
if we further require $2 |(\al+\be).$ 
\end{proof}

Consider the Euler operator, for $k\geq 0$,
{\small\begin{align*}
(E+1)(x_1x_2)^ke^{x_1x_2}&=(\frac{1}{2}\sum_i x_i\frac{\pa}{\pa x_i}+1)(x_1x_2)^ke^{x_1x_2}\\
&=\left((x_1x_2)^{k+1}+(k+1)(x_1x_2)^k\right)e^{x_1x_2}\in \hat{\gg}\cdot(Re^{x_1x_2}).
\end{align*}}
By induction we have
\[\big((x_1x_2)^{k}+(-1)^{k+1}k!\,\big)e^{x_1x_2}\in \hat{\gg}\cdot(Re^{x_1x_2}).\]

\begin{claim}\label{basis}For $X=\P^1$, at the LCSL we have a basis description
{\small\begin{align*} 
\hat{\gg}\cdot(Re^{x_1x_2})=&(\oplus_{k=1}^{\infty}\C\big((x_1x_2)^{k}+(-1)^{k+1}k!\,\big)e^{x_1x_2})\\
&\oplus(\oplus_{\alpha+\beta>0, \alpha\neq\beta,  2 |(\alpha+\beta)}\C x_1^\alpha x_2^\beta e^{x_1x_2})=:A.
\end{align*}}
\end{claim}
\begin{proof}
We already showed $\hat{\gg}\cdot(Re^{x_1x_2})\supset A$. 
It is clear that $e^{x_1x_2}\notin A$ and $A\oplus\C e^{x_1x_2}=Re^{x_1x_2}$, thus $\dim_\C Re^{x_1x_2}/A=1.$
Since
\[(Re^{f(b)}/\hat{\gg}\cdot(Re^{f(b)}))^*\simeq \cH om {}_{D^{\vee}}(\tau,\cO)_b\simeq \sol(\tau)_b \]
and we know in this case $\dim_\C \sol(\tau)_{b_\infty} =1,$ then $\dim_\C Re^{x_1x_2}/\hat{\gg}\cdot(Re^{x_1x_2})=1$. Therefore $\hat{\gg}\cdot(Re^{x_1x_2})= A$. 
\end{proof}

Now consider
{\small\begin{align*} 
\delta e^f(b_\infty)&=\left(\sum c_\al \left(\frac{\pa}{\pa a_0}\right)^{\al_0}\left(\frac{\pa}{\pa a_1}\right)^{\al_1}\left(\frac{\pa}{\pa a_2}\right)^{\al_2}\right)e^{f}(b_\infty)\\
&=\left(\sum c_\al x_1^{\al_0+2\al_1}x_2^{\al_0+2\al_2}\right)e^{x_1x_2}
\end{align*}}
Let $\delta e^f(b_\infty)\in \hat{\gg}\cdot(Re^{x_1x_2})$.
By Claim \ref{basis}, when $\al_1\neq \al_2,$ there is no restriction on $c_\alpha$;
when $\al_1=\al_2,$ let $d:=\al_0+\al_1+\al_2$, it forces 
{\small\[\left(\sum c_\al (x_1x_2)^{d}\right)e^{x_1x_2}=\bigg(\sum_{{d}\geq 1} c_\al \big((x_1x_2)^{d}+(-1)^{{d}+1}({d})!\big)\bigg)e^{x_1x_2}.\] }
Thus 
\[c_{0,0,0}=\sum_{{d}\geq 1} c_{\al_0,\al_1,\al_1}(-1)^{d+1}({d})!.\]
We can rewrite this as
\[\sum_{\al_1=\al_2,|\al|=d} c_{\al}(-1)^{d}d!=0.\]
Thus as a direct consequence of Theorem \ref{equiv}, we have:
\begin{claim}
When $X=\P^1$, if the coefficients of 
{\small$$\delta=\sum c_\al \left(\frac{\pa}{\pa a_0}\right)^{\al_0}\left(\frac{\pa}{\pa a_1}\right)^{\al_1}\left(\frac{\pa}{\pa a_2}\right)^{\al_2}$$}
 satisfy the linear relation
\[\sum_{\al_1=\al_2,\, |\al|=d} c_{\al}(-1)^{d}d!=0,\]
then $\delta\gs(b_\infty)=0$ for all $\gs\in \sol(\tau)_b.$
\end{claim}

\bs
\textbf{Computation for $\P^2$ at LCSL.}
For $X=\P^2,\, G=SL_3,\,R\equiv\C[x_\alpha,|\al|=3]$, $f=a_0x_1x_2x_3+a_1x^3_1+a_2x_1^2x_2+a_3x_1x_2^2+a_4x_2^3+a_5x_2^2x_3+a_6x_2x_3^2+a_7x_3^3+a_8x_1x_3^2+a_9x^2_1x_3$, $b_\infty=x_1x_2x_3$.

Similar to the $\P^1$ case, we can show
\begin{claim}\label{basisLP2}For $X=\P^2$, at the LCSL we have a basis description
{\small\begin{align*}
\hat{\gg}\cdot(Re^{x_1x_2x_3})=&(\oplus_{k=1}^{\infty}\C\big((x_1x_2x_3)^{k}+(-1)^{k+1}k!\big)e^{x_1x_2x_3})\\
&\oplus(\oplus_{\iota_1+\iota_2+\iota_3>0, \iota_1,\iota_2,\iota_3\text{ not all equal, } 3|(\iota_1+\iota_2+\iota_3)}\C x_1^{\iota_1} x_2^{\iota_2}x_3^{\iota_3}e^{x_1x_2x_3}).
\end{align*}}
\end{claim}
In this case we know $\dim_\C \sol(\tau)_{b_\infty} =1$, so $\hat{\gg}\cdot(Re^{x_1x_2x_3})$ is of codimension $1$.

Now consider 
{\small\begin{align*}
&\delta e^f(b_\infty)=(\sum c_\alpha(\frac{\partial}{\partial a_0})^{\alpha_0}\cdots(\frac{\partial}{\partial a_9})^{\alpha_9})e^{f}(b_\infty)\\
=&(\sum c_\al x_1^{\al_0+3\al_1+2\al_2+\al_3+\al_8+2\al_9}x_2^{\al_0+\al_2+2\al_3+3\al_4+2\al_5+\al_6}x_3^{\al_0+\al_5+2\al_6+3\al_7+2\al_8+\al_9})\\
&\cdot e^{x_1x_2x_3}.
\end{align*}}
Let 
{\small\begin{equation}\label{P2Fermat}
\begin{cases}
\be_1:=&\al_0+3\al_1+2\al_2+\al_3+\al_8+2\al_9,\\
\be_2:=&\al_0+\al_2+2\al_3+3\al_4+2\al_5+\al_6,\\
\be_3:=&\al_0+\al_5+2\al_6+3\al_7+2\al_8+\al_9.
\end{cases}
\end{equation}}
By Claim \ref{basisLP2} we can see that there is no restriction on the coefficient $c_\al$ unless $$\be_1=\be_2=\be_3=|\al|.$$
Let $\delta e^f(b_\infty)\in \hat{\gg}\cdot(Re^{x_1x_2x_3})$, it forces 
\[\sum_{\be_1=\be_2=\be_3=d} c_\al (-1)^d(d!)=0.\]
Thus by Theorem \ref{equiv}, we have:
\begin{claim}
When $X=\P^2$, if the coefficients of  $\delta=\sum c_\al \left(\frac{\pa}{\pa a_0}\right)^{\al_0}\cdots\left(\frac{\pa}{\pa a_9}\right)^{\al_9}$
satisfy the linear relation
\[\sum_{\be_1=\be_2=\be_3=d} c_\al (-1)^d(d!)=0,\]
then $\delta\gs(b_\infty)=0$ for all $\gs\in \sol(\tau)_b.$
\end{claim}

\bs
\textbf{Computation for $\P^1$ at the Fermat point.}
$X=\P^1,\, G=SL_2,\, b_F=x_1^2+x_2^2.$

Let $(-1)!!=1.$ By straightforward induction which we omit here, we can show
\begin{claim}For $X=\P^1$, at the Fermat point we have a basis description
{\small\begin{align*}
\hat{\gg}\cdot(Re^{x_1^2+x_2^2})=&\left(\oplus_{k\equiv l\equiv 1(\text{mod }2)}\C x_1^{k}x_2^le^{x_1^2+x_2^2}\right)
\oplus\\
&\left(\oplus_{k\equiv l\equiv 0(\text{mod }2),k+l\geq 2}\C\left(x_1^{k}x_2^l-(-1)^{\frac{k+l}{2}}\frac{(k-1)!!(l-1)!!}{2^{(k+l)/2}}\right)e^{x_1^2+x_2^2}\right).
\end{align*}}
\end{claim}
In this case we know $\dim_\C \sol(\tau)_{b_F} =1$ and $\hat{\gg}\cdot(Re^{x_1^2+x_2^2})$ is of codimension $1$.

And by Theorem \ref{equiv}, we have:
\begin{claim}
When $X=\P^1$, if the coefficients of  
{\small$$\delta=\sum c_\al \left(\frac{\pa}{\pa a_0}\right)^{\al_0}\left(\frac{\pa}{\pa a_1}\right)^{\al_1}\left(\frac{\pa}{\pa a_2}\right)^{\al_2}$$}
satisfy the linear relation
{\small\[\sum_{\alpha_0\equiv 0\mod 2}c_{\al_0,\al_1,\al_2}(-1)^{\alpha_0+\al_1+\al_2}\frac{({\alpha_0+2\al_1}-1)!!({\alpha_0+2\al_2}-1)!!}{2^{(\alpha_0+\al_1+\al_2)}}
=0,\]}
then $\delta\gs(b_F)=0$ for all $\gs\in \sol(\tau)_b.$
\end{claim}

\bs
\textbf{Computation for $\P^2$ at the Fermat point.}
$X=\P^2,\, G=SL_3,\, b_F=x_1^3+x_2^3+x_3^3.$

Let $(-1)!!!=(-2)!!!=1$. By straightforward induction, we can show
\begin{claim}Let $\iota_0+\iota_1+\iota_2:=c$.
For $X=\P^2$, at the Fermat point we have a basis description\\
{\small$\hat{\gg}\cdot(Re^{x_0^3+x_1^3+x_2^3})=\left(\oplus_{\text{one of }\iota_i\equiv 2(\mod 3)}\C x_0^{\iota_0}x_1^{\iota_1}x_2^{\iota_2}e^{x_0^3+x_1^3+x_2^3}\right)
\oplus\\
\left(\oplus_{\iota_0\equiv\iota_1\equiv\iota_2\equiv 0(\mod 3),c\geq 3}\C\left(x_0^{\iota_0}x_1^{\iota_1}x_2^{\iota_2}-(-1)^{\frac{c}{3}}\frac{(\iota_0-2)!!!(\iota_1-2)!!!(\iota_2-2)!!!}{3^\frac{c}{3}}\right)e^{x_0^3+x_1^3+x_2^3}\right)\oplus\\
\left(\oplus_{\iota_0\equiv\iota_1\equiv\iota_2\equiv 1(\mod 3),c\geq 6}\C\left(x_0^{\iota_0}x_1^{\iota_1}x_2^{\iota_2}+(-1)^{\frac{c}{3}}\frac{(\iota_0-2)!!!(\iota_1-2)!!!(\iota_2-2)!!!}{3^{\frac{c}{3}-1}}x_0x_1x_2\right)e^{x_0^3+x_1^3+x_2^3}\right).$}
\end{claim}
In this case we know $\dim_\C \sol(\tau)_{b_F} =2$ and $\hat{\gg}\cdot(Re^{x_1^3+x_2^3+x_3^3})$ is of codimension $2$.

Then by Theorem \ref{equiv}, we have:
\begin{claim}
When $X=\P^2$, if the coefficients of  $\delta=\sum c_\al \left(\frac{\pa}{\pa a_0}\right)^{\al_0}\cdots\left(\frac{\pa}{\pa a_9}\right)^{\al_9}$
satisfy the linear relation
{\small\[
\begin{cases}
&\sum_{\be_1\equiv\be_2\equiv\be_3\equiv 0\mod 3}c_\al(-1)^{\frac{\be_1+\be_2+\be_3}{3}}\dfrac{(\be_1-2)!!!(\be_2-2)!!!(\be_3-2)!!!}{3^{\frac{\be_1+\be_2+\be_3}{3}}}=0,\\
&\sum_{\be_1\equiv\be_2\equiv\be_3\equiv 1\mod 3}c_\al(-1)^{\frac{\be_1+\be_2+\be_3}{3}-1}\dfrac{(\be_1-2)!!!(\be_2-2)!!!(\be_3-2)!!!}{3^{\frac{\be_1+\be_2+\be_3}{3}-1}}=0
\end{cases}
\]}
where $\be_i$ are defined in \eqref{P2Fermat},
then $\delta\gs(b_F)=0$ for all $\gs\in \sol(\tau)_b.$
\end{claim}

\section{Expressions of $S$ and $T$}\label{appendix B}
The degree $4$ invariant $S$ of a ternary cubic equals (see \cite[p.167]{S})
{\small\begin{align*}
S=&- a_{300}a_{012}^2a_{120} + a_{012}^2a_{210}^2 + a_{300}a_{012}a_{021}a_{111}- a_{012}a_{021}a_{201}a_{210} - a_{012}a_{102}a_{120}a_{210} \\
&+ a_{030}a_{300}a_{012}a_{102} - 2a_{012}a_{111}^2a_{210} + 3a_{012}a_{111}a_{120}a_{201} - a_{030}a_{012}a_{201}^2 - a_{300}a_{021}^2a_{102} \\
&+ a_{021}^2a_{201}^2 + 3a_{021}a_{102}a_{111}a_{210} - a_{021}a_{102}a_{120}a_{201} - 2a_{021}a_{111}^2a_{201} + a_{003}a_{300}a_{021}a_{120} \\
&- a_{003}a_{021}a_{210}^2 + a_{102}^2a_{120}^2 - a_{030}a_{102}^2a_{210} - 2a_{102}a_{111}^2a_{120} + a_{030}a_{102}a_{111}a_{201} + a_{111}^4 \\
&+ a_{003}a_{111}a_{120}a_{210} - a_{003}a_{030}a_{300}a_{111} - a_{003}a_{120}^2a_{201} + a_{003}a_{030}a_{201}a_{210}.
\end{align*}
}

The degree $6$ invariant $T$ of the ternary cubic equals (see \cite[p.171]{S})
{\small\begin{align*}
T=& a_{003}^2a_{030}^2a_{300}^2 - 6a_{003}^2a_{030}a_{120}a_{210}a_{300} + 4a_{003}^2a_{030}a_{210}^3 + 4a_{003}^2a_{120}^3a_{300}  \\
&- 6a_{003}a_{012}a_{021}a_{030}a_{300}^2 + 18a_{003}a_{012}a_{021}a_{120}a_{210}a_{300} - 12a_{003}a_{012}a_{021}a_{210}^3 \\
&+ 12a_{003}a_{012}a_{030}a_{111}a_{210}a_{300} + 6a_{003}a_{012}a_{030}a_{120}a_{201}a_{300} - 12a_{003}a_{012}a_{030}a_{201}a_{210}^2 \\
&- 24a_{003}a_{012}a_{111}a_{120}^2a_{300} + 12a_{003}a_{012}a_{111}a_{120}a_{210}^2 + 6a_{003}a_{012}a_{120}^2a_{201}a_{210}\\
&- 24a_{003}a_{021}^2a_{111}a_{210}a_{300} - 12a_{003}a_{021}^2a_{120}a_{201}a_{300}+ 24a_{003}a_{021}^2a_{201}a_{210}^2 \\
&+ 6a_{003}a_{021}a_{030}a_{102}a_{210}a_{300}+ 12a_{003}a_{021}a_{030}a_{111}a_{201}a_{300} - 12a_{003}a_{021}a_{030}a_{201}^2a_{210}\\
&- 12a_{003}a_{021}a_{102}a_{120}^2a_{300} + 6a_{003}a_{021}a_{102}a_{120}a_{210}^2 + 36a_{003}a_{021}a_{111}^2a_{120}a_{300}\\
&+ 12a_{003}a_{021}a_{111}^2a_{210}^2- 60a_{003}a_{021}a_{111}a_{120}a_{201}a_{210} + 24a_{003}a_{021}a_{120}^2a_{201}^2 \\
&- 6a_{003}a_{030}^2a_{102}a_{201}a_{300}  + 4a_{003}a_{030}^2a_{201}^3 + 12a_{003}a_{030}a_{102}a_{111}a_{120}a_{300}\\
&- 24a_{003}a_{030}a_{102}a_{111}a_{210}^2 + 18a_{003}a_{030}a_{102}a_{120}a_{201}a_{210} - 20a_{003}a_{030}a_{111}^3a_{300} \\
&+ 36a_{003}a_{030}a_{111}^2a_{201}a_{210} - 24a_{003}a_{030}a_{111}a_{120}a_{201}^2  + 12a_{003}a_{102}a_{111}a_{120}^2a_{210}\\
& - 12a_{003}a_{102}a_{120}^3a_{201}  - 12a_{003}a_{111}^3a_{120}a_{210} + 12a_{003}a_{111}^2a_{120}^2a_{201} + 4a_{012}^3a_{030}a_{300}^2\\
& - 12a_{012}^3a_{120}a_{210}a_{300} + 8a_{012}^3a_{210}^3 - 3a_{012}^2a_{021}^2a_{300}^2 + 12a_{012}^2a_{021}a_{111}a_{210}a_{300} \\
&+ 6a_{012}^2a_{021}a_{120}a_{201}a_{300} - 12a_{012}^2a_{021}a_{201}a_{210}^2 -  12a_{012}^2a_{030}a_{102}a_{210}a_{300} \\
&- 24a_{012}^2a_{030}a_{111}a_{201}a_{300} + 24a_{012}^2a_{030}a_{201}^2a_{210} + 24a_{012}^2a_{102}a_{120}^2a_{300}\\
&+ 12a_{012}^2a_{111}^2a_{120}a_{300} - 24a_{012}^2a_{111}^2a_{210}^2 + 36a_{012}^2a_{111}a_{120}a_{201}a_{210} - 27a_{012}^2a_{120}^2a_{201}^2\\
& + 6a_{012}a_{021}^2a_{102}a_{210}a_{300} + 12a_{012}a_{021}^2a_{111}a_{201}a_{300} - 12a_{012}a_{021}^2a_{201}^2a_{210}\\
& + 18a_{012}a_{021}a_{030}a_{102}a_{201}a_{300} - 12a_{012}a_{021}a_{030}a_{201}^3 - 60a_{012}a_{021}a_{102}a_{111}a_{120}a_{300}\\
& + 36a_{012}a_{021}a_{102}a_{111}a_{210}^2 - 6a_{012}a_{021}a_{102}a_{120}a_{201}a_{210} - 12a_{012}a_{021}a_{111}^3a_{300}\\
&- 12a_{012}a_{021}a_{111}^2a_{201}a_{210} + 36a_{012}a_{021}a_{111}a_{120}a_{201}^2 - 12a_{012}a_{030}a_{102}^2a_{120}a_{300}\\
&+ 24a_{012}a_{030}a_{102}^2a_{210}^2 + 36a_{012}a_{030}a_{102}a_{111}^2a_{300} - 60a_{012}a_{030}a_{102}a_{111}a_{201}a_{210}\\
&+ 6a_{012}a_{030}a_{102}a_{120}a_{201}^2 + 12a_{012}a_{030}a_{111}^2a_{201}^2 - 12a_{012}a_{102}^2a_{120}^2a_{210} \\
&+ 36a_{012}a_{102}a_{111}a_{120}^2a_{201} + 24a_{012}a_{111}^4a_{210} - 36a_{012}a_{111}^3a_{120}a_{201}\\
& + 8a_{021}^3a_{201}^3 + 24a_{021}^2a_{102}^2a_{120}a_{300} - 27a_{021}^2a_{102}^2a_{210}^2 + 12a_{021}^2a_{102}a_{111}^2a_{300} \\
&+ 36a_{021}^2a_{102}a_{111}a_{201}a_{210}   - 12a_{021}^2a_{102}a_{120}a_{201}^2 - 24a_{021}^2a_{111}^2a_{201}^2 \\
&+ 6a_{021}a_{030}a_{102}^2a_{201}a_{210} + 12a_{021}a_{030}a_{102}a_{111}a_{201}^2 + 36a_{021}a_{102}^2a_{111}a_{120}a_{210} \\
&- 12a_{021}a_{102}^2a_{120}^2a_{201}- 36a_{021}a_{102}a_{111}^3a_{210} - 12a_{021}a_{102}a_{111}^2a_{120}a_{201}  \\
& + 4a_{030}^2a_{102}^3a_{300} - 3a_{030}^2a_{102}^2a_{201}^2 - 12a_{030}a_{102}^3a_{120}a_{210} + 12a_{030}a_{102}^2a_{111}^2a_{210}\\
&+ 12a_{030}a_{102}^2a_{111}a_{120}a_{201} - 12a_{030}a_{102}a_{111}^3a_{201} + 8a_{102}^3a_{120}^3 - 24a_{102}^2a_{111}^2a_{120}^2 \\
&- 3a_{003}^2a_{120}^2a_{210}^2+ 4a_{003}a_{021}^3a_{300}^2 - 12a_{012}^2a_{102}a_{120}a_{210}^2 - 12a_{012}a_{102}a_{111}^2a_{120}a_{210} \\
&- 24a_{021}a_{030}a_{102}^2a_{111}a_{300}+ 24a_{021}a_{111}^4a_{201} - 12a_{021}^3a_{102}a_{201}a_{300} \\
&+ 24a_{102}a_{111}^4a_{120} - 8a_{111}^6.
\end{align*}
}

\noindent\address {\SMALL J. Chen, Yau Mathematical Sciences Center, Tsinghua University, Haidian District, Beijing 100084, China\\ jychen@brandeis.edu.}
\vskip-.15in

\noindent\address {\SMALL A. Huang, Department of Mathematics, Brandeis University, Waltham MA 02454, U.S.A. \\ anhuang@brandeis.edu.}
\vskip-.15in

\noindent\address {\SMALL B.H. Lian, Department of Mathematics, Brandeis University, Waltham MA 02454, U.S.A.\\ lian@brandeis.edu.}
\vskip-.15in

\noindent\address {\SMALL S.-T. Yau, Department of Mathematics, Harvard University, Cambridge MA 02138, U.S.A. \\ yau@math.harvard.edu.}

\end{document}